\documentclass[fleqn, pdftex]{article} 

\usepackage{amsmath,amssymb,amsthm,pifont,graphics,tikz}

\title{Some observations on the FGH theorem}
\author{Taishi Kurahashi\footnote{Email: kurahashi@people.kobe-u.ac.jp}
\footnote{Graduate School of System Informatics, Kobe University, 1-1 Rokkodai, Nada, Kobe 657-8501, Japan.}}
\date{}

\theoremstyle{plain}
\newtheorem{thm}{Theorem}[section]
\newtheorem{lem}[thm]{Lemma}
\newtheorem{prop}[thm]{Proposition}
\newtheorem{cor}[thm]{Corollary}
\newtheorem{fact}[thm]{Fact}
\newtheorem{prob}[thm]{Problem}

\theoremstyle{definition}
\newtheorem{defn}[thm]{Definition}

\newcommand{\GL}{\mathbf{GL}}
\newcommand{\GLS}{\mathbf{S}}
\newcommand{\PA}{\mathsf{PA}}
\newcommand{\PR}{\mathrm{Pr}}
\newcommand{\Prf}{\mathrm{Prf}}
\newcommand{\Prov}{\mathrm{Prov}}
\newcommand{\Proof}{\mathrm{Proof}}
\newcommand{\Con}{\mathrm{Con}}
\newcommand{\gn}[1]{\ulcorner#1\urcorner}
\newcommand{\pair}[1]{\langle#1\rangle}
\newcommand{\lp}[1]{\overline{\langle#1\rangle}}
\newcommand{\Sub}{\mathsf{Sub}}
\newcommand{\N}{\mathbb{N}}
\newcommand{\num}{\overline}

\newcommand{\LA}{\mathcal{L}_A}

\newcommand{\FGH}{\mathrm{FGH}}
\newcommand{\sFGH}{\mathsf{FGH}}
\newcommand{\Rf}{\mathrm{Rfn}}
\newcommand{\cx}{\mathfrak{c}}

\newcommand{\pc}{\preccurlyeq}
\newcommand{\R}{\scriptsize{\reflectbox{{\rm R}}}}

\begin{document}

\maketitle

\begin{abstract}
We investigate the Friedman--Goldfarb--Harrington theorem from two perspectives. 
Firstly, in the frameworks of classical and modal propositional logics, we study the forms of sentences whose existence is guaranteed by the FGH theorem. 
Secondly, we prove some variations of the FGH theorem with respect to Rosser provability predicates. 
\end{abstract}

\section{Introduction}

The notion of the weak representability of computably enumerable (c.e.)~sets plays an important role in a proof of G\"odel's first incompleteness theorem. 
We say that a set $X$ of natural numbers is \textit{weakly representable} in a theory $T$ if there exists a formula $\varphi_X(v)$ such that for any natural number $n$, $n \in X$ if and only if $T \vdash \varphi_X(\num{n})$. 
If a set $X$, which is c.e.~but not computable, is weakly representable in a c.e.~theory $T$, then it is shown that there exists a natural number $n$ such that $T \nvdash \varphi_X(\num{n})$ and $T \nvdash \neg \varphi_X(\num{n})$. 
Therefore, such a  theory $T$ is incomplete. 

It is easily shown that every c.e.~set is weakly representable in every $\Sigma_1$-sound c.e.~extension of $\PA$ because of $\Sigma_1$-soundness and $\Sigma_1$-completeness. 
Furthermore, Ehrenfeucht and Feferman \cite{EF} proved the weak representability of c.e.~sets in every consistent c.e.~extension of $\PA$ (see also Shepherdson \cite{She}). 

\begin{thm}[Ehrenfeucht and Feferman]\label{EFThm}
Let $T$ be any consistent c.e.~extension of $\PA$. 
Then, every c.e.~set of natural numbers is weakly representable in $T$. 
\end{thm}




Ehrenfeucht and Feferman's theorem is a metamathematical result, but is formalizable and also provable in $\PA$. 
This fact is called the \textit{FGH theorem} (see Smory\'nski \cite[p.366]{Smo}, Visser \cite[Theorem 4.1]{Vis84}, \cite[Section 3]{Vis05} and Lindstr\"om \cite[Exercise 2.26 (a)]{Lin}). 

\begin{thm}[The FGH theorem]
Let $T$ be any consistent c.e.~extension of $\PA$. 
Then, for any $\Sigma_1$ sentence $\sigma$, there exists a sentence $\varphi$ such that
\[
	\PA + \Con_T \vdash \sigma \leftrightarrow \PR_T(\gn{\varphi}).
\] 
\end{thm}

Here $\PR_T(x)$ is a natural $\Sigma_1$ provability predicate of $T$ which expresses the provability of $T$, and $\Con_T$ is the $\Pi_1$ sentence $\neg \PR_T(\gn{0=1})$ expressing the consistency of $T$. 
``FGH'' stands for Friedman, Goldfarb and Harrington. 
Harrington and Friedman pointed out that $\varphi$ in the statement is found as $\Pi_1$ and $\Sigma_1$, respectively. 
For a history of the FGH theorem, see Visser \cite{Vis05}. 
The FGH theorem has been used in the literature (for example, in papers by the author, it appears in \cite{KK,Kur18}).


A version of the FGH theorem with a parameter is also proved. 
That is, by modifying the usual proof of the FGH theorem, it is proved that for any $\Sigma_1$ formula $\sigma(v)$ with the only free variable $v$, there exists a formula $\varphi(v)$ such that $\PA + \Con_T$ proves $\forall v\, \bigl(\sigma(v) \leftrightarrow \PR_T(\gn{\varphi(\dot{v})}) \bigr)$. 
Here $\gn{\varphi(\dot{v})}$ is a primitive recursive term corresponding to a primitive recursive function calculating the G\"odel number of $\varphi(\num{n})$ from $n$. 
If $T$ is consistent, then the theory $\PA + \Con_T$ is sound, and hence the weak representability of c.e.~sets in $T$ follows from this parameterized version of the FGH theorem. 

In the present paper, we analyze the FGH theorem from two perspectives. 
The first perspective is the ``form'' of the sentence $\varphi$ in the statement of the FGH theorem. 
Let $\mathrm{True}_{\Sigma_1}(v)$ be a partial truth definition for $\Sigma_1$ sentences, that is, $\mathrm{True}_{\Sigma_1}(v)$ is a $\Sigma_1$ formula such that for any $\Sigma_1$ sentence $\sigma$, $\PA$ proves $\sigma \leftrightarrow \mathrm{True}_{\Sigma_1}(\gn{\sigma})$ (cf.~Lindstr\"om \cite[Fact 10]{Lin}). 
By the parameterized version of the FGH theorem, there exists a formula $\xi(v)$ such that $\PA + \Con_T \vdash \forall v \, \bigl(\mathrm{True}_{\Sigma_1}(v) \leftrightarrow \PR_T(\gn{\xi(\dot{v})}) \bigr)$. 
Then, for any $\Sigma_1$ sentence $\sigma$, $\PA + \Con_T \vdash \sigma \leftrightarrow \PR_T(\gn{\xi(\gn{\sigma})})$. 
Therefore, $\varphi$ in the FGH theorem can be taken in the form $\xi(\cdot)$ uniformly regardless of $\sigma$. 

From the above observation, we consider the following question: What form of the sentence $\varphi$ in the FGH theorem can we find?
In Section \ref{Classical}, we investigate this question in the framework of classical propositional logic. 
Specifically, we study the following rephrased question: What is a propositional formula $A$ such that $\varphi$ in the FGH theorem can be taken uniformly in the form $A$ regardless of $\sigma$?
In Section \ref{Modal}, we investigate this rephrased question in the framework of modal propositional logic. 
However, unlike the case of classical propositional logic, the formula $A$ may contain the modal operator $\Box$. 
Therefore, the interpretation of $\Box$ in arithmetic should be clearly defined. 
As is usually done in the study of provability logic, in the present paper, $\Box$ is interpreted by $\PR_T(x)$. 
Then, our proofs of theorems in Section \ref{Modal} are modifications of that of Solovay's arithmetical completeness theorem \cite{Sol} which is one of the most remarkable theorems in the research of provability logic. 

The second perspective is the choice of provability predicates. 
Joosten \cite{Joo} generalized the FGH theorem by proving similar statements concerning several nonstandard provability predicates such as a formula expressing the provability in $T$ together with all true $\Sigma_{n+2}$ sentences. 
In the last section, we study some variations of the FGH theorem with respect to Rosser provability predicates. 

Sections \ref{Classical}, \ref{Modal}, and \ref{Rosser} can be read independently.
We close the introduction with common preparations for reading these sections.
Let $\LA$ denote the language of first-order arithmetic containing the symbols $0, S, +, \times, \leq$. 
We do not specify what exactly $\LA$ is, but it may be assumed to have as many function symbols for primitive recursive functions as necessary. 
Throughout the present paper, we fix a consistent c.e.~extension $T$ of Peano Arithmetic $\PA$ in the language $\LA$. 
Let $\omega$ denote the set of all natural numbers. 
For each $n \in \omega$, let $\num{n}$ denote the numeral $S(S(\cdots S(0) \cdots))$ ($n$ times applications of $S$) for $n$. 
We fix a primitive recursive formula $\Proof_T(x, y)$ naturally expressing that ``$y$ is a $T$-proof of $x$''. 
Our $\Sigma_1$ provability predicate $\PR_T(x)$ of $T$ is defined by $\exists y\, \Proof_T(x, y)$, saying that ``$x$ is $T$-provable''. 
Then, we may assume that the provability predicate $\PR_T(x)$ satisfies the following clauses (see Boolos \cite{Boo}): 
\begin{itemize}
	\item If $T \vdash \varphi$, then $\PA \vdash \PR_T(\gn{\varphi})$, 
	\item $\PA \vdash \PR_T(\gn{\varphi \to \psi}) \to \bigl(\PR_T(\gn{\varphi}) \to \PR_T(\gn{\psi}) \bigr)$, 
	\item If $\varphi$ is a $\Sigma_1$ sentence, then $\PA \vdash \varphi \to \PR_T(\gn{\varphi})$. 
\end{itemize}
We may also assume that $\PA$ verifies that every theorem of $T$ has infinitely many proofs, that is, $\PA \vdash \forall x \forall y \, \bigl(\Proof_T(x, y) \to \exists z\, {>}\, y\, \Proof_T(x, z) \bigr)$. 

We introduce witness comparison notation (cf.~Lindstr\"om \cite[Lemma 1.3]{Lin}). 

\begin{defn}[Witness comparison notation]\label{WC}
Suppose that $\LA$-formulas $\varphi$ and $\psi$ are of the forms $\exists x \, \varphi'(x)$ and $\exists y \, \psi'(y)$, respectively. 
\begin{itemize}
	\item $\varphi \pc \psi$ is an abbreviation for $\exists x\, \bigl(\varphi'(x) \land \forall y\, {<}\, x \, \neg \psi'(y) \bigr)$.  
	\item $\varphi \prec \psi$ is an abbreviation for $\exists x\, \bigl(\varphi'(x) \land \forall y\, {\leq}\, x \, \neg \psi'(y) \bigr)$.  
\end{itemize}
\end{defn}

It is easily verified that $\PA$ proves the formulas $\neg (\varphi \pc \psi \land \psi \prec \varphi)$ and $\varphi \lor \psi \to (\varphi \pc \psi \lor \psi \prec \varphi)$.

\section{On the form of a sentence in the FGH theorem\\ -- the case of classical propositional logic}\label{Classical}

In this section, we investigate the following question mentioned in the introduction: What is a propositional formula $A$ such that $\varphi$ in the FGH theorem can be taken uniformly in the form $A$ regardless of $\sigma$?

The language of classical propositional logic consists of countably many propositional variables $p_1, p_2, \ldots, q_1, q_2, \ldots$, propositional constants $\top, \bot$, and propositional connectives $\land, \lor, \neg, \to, \leftrightarrow$. 
For each propositional formula $A$, let $\models A$ mean that $A$ is a tautology. 
We say $A$ is unsatisfiable if $\models \neg A$. 
We say that a propositional formula is \textit{contingent} if it is neither a tautology nor unsatisfiable. 

For any propositional formula $A(p_1, \ldots, p_n)$ containing only the indicated propositional variables and any $\LA$-sentences $\varphi_1, \ldots, \varphi_n$, let $A(\varphi_1, \ldots, \varphi_n)$ denote the $\LA$-sentence obtained by simultaneously replacing all the occurrences of $p_i$ in $A$ by $\varphi_i$, for each $i \in \{1, \ldots, n\}$.

We introduce the following sets in order to simplify our descriptions. 

\begin{defn}
For each $\Sigma_1$ sentence $\sigma$, let $\FGH_T(\sigma)$ be the set of all $\LA$-sentences $\varphi$ such that $\PA + \Con_T \vdash \sigma \leftrightarrow \Pr_T(\gn{\varphi})$. 
\end{defn}

Then, the FGH theorem states that for any $\Sigma_1$ sentence $\sigma$, the set $\FGH_T(\sigma)$ is non-empty. 
We show that the set $\FGH(\sigma)$ is closed under the $T$-provable equivalence. 

\begin{prop}\label{FGHeq}
For any $\Sigma_1$ sentence $\sigma$ and any $\LA$-sentences $\varphi$ and $\psi$, if $\varphi \in \FGH_T(\sigma)$ and $T \vdash \varphi \leftrightarrow \psi$, then $\psi \in \FGH_T(\sigma)$. 
\end{prop}
\begin{proof}
Suppose $\varphi \in \FGH_T(\sigma)$ and $T \vdash \varphi \leftrightarrow \psi$. 
Then, the equivalence $\PR_T(\gn{\varphi}) \leftrightarrow \PR_T(\gn{\psi})$ is provable in $\PA$. 
Since $\PA + \Con_T \vdash \sigma \leftrightarrow \PR_T(\gn{\varphi})$, we obtain $\PA + \Con_T \vdash \sigma \leftrightarrow \PR_T(\gn{\psi})$, and hence $\psi \in \FGH_T(\sigma)$. 
\end{proof}

First, we prove the following introductory theorem. 

\begin{thm}\label{Thm1}
Let $A(p_1, \ldots, p_n)$ be any propositional formula with only the indicated propositional variables. 
Then, the following are equivalent: 
\begin{enumerate}
	\item $A(p_1, \ldots, p_n)$ is contingent. 
	\item For any $\Sigma_1$ sentence $\sigma$, there exist $\LA$-sentences $\varphi_1, \ldots, \varphi_n$ such that $A(\varphi_1, \ldots, \varphi_n) \in \FGH_T(\sigma)$. 
\end{enumerate}
\end{thm}

For each formula $A$, let $A^0$ and $A^1$ be $\neg A$ and $A$, respectively. 
Theorem \ref{Thm1} follows from the following lemma. 

\begin{lem}\label{Lem1}
Let $A(p_1, \ldots, p_n)$ be any propositional formula with only the indicated propositional variables and let $r$ be any propositional variable not contained in $A$. 
If $A(p_1, \ldots, p_n)$ is contingent, then there exist propositional formulas $B_1(r), \ldots, B_n(r)$ such that $\models r \leftrightarrow A(B_1(r), \ldots, B_n(r))$. 
\end{lem}
\begin{proof}
Suppose that $A(p_1, \ldots, p_n)$ is contingent. 
Let $f$ and $g$ be mappings from $\{1, \ldots, n\}$ to $\{0, 1\}$ such that $\models A(\top^{f(1)}, \ldots, \top^{f(n)})$ and $\models \neg A(\top^{g(1)}, \ldots, \top^{g(n)})$. 
For each $i$ ($1 \leq i \leq n$), let $B_i(r)$ be the propositional formula $(r \land \top^{f(i)}) \lor (\neg r \land \top^{g(i)})$. 
Then, $\models r \to (B_i(r) \leftrightarrow \top^{f(i)})$ and $\models \neg r \to (B_i(r) \leftrightarrow \top^{g(i)})$ hold. 
Thus, we have 
\[
	\models r \to \bigl(A(B_1(r), \ldots, B_n(r)) \leftrightarrow A(\top^{f(1)}, \ldots, \top^{f(n)}) \bigr)
\]
and
\[
	\models \neg r \to \bigl(A(B_1(r), \ldots, B_n(r)) \leftrightarrow A(\top^{g(1)}, \ldots, \top^{g(n)}) \bigr). 
\]
Since $\models A(\top^{f(1)}, \ldots, \top^{f(n)})$ and $\models \neg A(\top^{g(1)}, \ldots, \top^{g(n)})$, we obtain 
\[
	\models r \to A(B_1(r), \ldots, B_n(r))\ \text{and}\ \models \neg r \to \neg A(B_1(r), \ldots, B_n(r)). \tag*{\mbox{\qedhere}}
\] 
\end{proof}

\begin{proof}[Proof of Theorem \ref{Thm1}]
$(1 \Rightarrow 2)$: 
Suppose that $A(p_1, \ldots, p_n)$ is contingent and let $\sigma$ be any $\Sigma_1$ sentence. 
Then, by Lemma \ref{Lem1}, there exist propositional formulas $B_1(r), \ldots, B_n(r)$ such that 
\begin{equation}\label{equiv0}
	\models r \leftrightarrow A(B_1(r), \ldots, B_n(r)).
\end{equation} 
By the FGH theorem, there exists an $\LA$-sentence $\varphi \in \FGH_T(\sigma)$. 
For each $i$ ($1 \leq i \leq n$), let $\varphi_i$ be the $\LA$-sentence $B_i(\varphi)$. 
Then, by the equivalence (\ref{equiv0}), we obtain
\[
	\PA \vdash \varphi \leftrightarrow A(\varphi_1, \ldots, \varphi_n).
\] 
By Proposition \ref{FGHeq}, we conclude $A(\varphi_1, \ldots, \varphi_n) \in \FGH_T(\sigma)$ by Proposition \ref{FGHeq}. 

$(2 \Rightarrow 1)$: 
Suppose that $A(p_1, \ldots, p_n)$ is not contingent. 
Then, for any $\LA$-sentences $\varphi_1, \ldots, \varphi_n$, either $A(\varphi_1, \ldots, \varphi_n)$ or $\neg A(\varphi_1, \ldots, \varphi_n)$ is provable in $\PA$. 
Let $\sigma$ be any $\Sigma_1$ sentence independent of $\PA + \Con_T$. 
Then, $A(\varphi_1, \ldots, \varphi_n)$ is not in $\FGH_T(\sigma)$ for all $\LA$-sentences $\varphi_1, \ldots, \varphi_n$. 
\end{proof}

As mentioned in the introduction, the FGH theorem and the weak representability theorem are related to each other. 
In particular, we can also prove the following theorem by a similar proof.

\begin{thm}\label{Thm1'}
Let $A(p_1, \ldots, p_n)$ be any propositional formula with only the indicated propositional variables. 
Then, the following are equivalent: 
\begin{enumerate}
	\item $A(p_1, \ldots, p_n)$ is contingent. 
	\item For any c.e.~set $X$, there exist $\LA$-formulas $\varphi_1(v), \ldots, \varphi_n(v)$ such that $A(\varphi_1(v), \ldots, \varphi_n(v))$ weakly represents $X$ in $T$. 
\end{enumerate}
\end{thm}

Theorem \ref{Thm1} will be extended to the framework of modal propositional logic in the next section. 
In this section, we further improve Theorem \ref{Thm1} in the framework of classical propositional logic. 
For any propositional formula $A(p_1, \ldots, p_n, q_1, \ldots, q_m)$ with the only indicated propositional variables, let $\sFGH_T[A; q_1, \ldots, q_m]$ denote the metamathematical statement ``for any $\LA$-sentences $\psi_1, \ldots, \psi_m$ and for any $\Sigma_1$ sentence $\sigma$, there exist $\LA$-sentences $\varphi_1, \ldots, \varphi_n$ such that $A(\varphi_1, \ldots, \varphi_n, \psi_1, \ldots, \psi_m) \in \FGH_T(\sigma)$'', and we provide a necessary and sufficient condition for $\sFGH_T[A; q_1, \ldots, q_m]$. 
From this, we obtain more detailed information about the elements of $\FGH(\sigma)$ and the first incompleteness theorem (see Corollary \ref{FI1}).

Let $\mathcal{F}_m$ denote the set of all functions $f : \{1, \ldots, m\} \to \{0, 1\}$. 
We prove the following theorem which is one of main theorems of the present paper.

\begin{thm}\label{Thm2}
For any propositional formula $A(p_1, \ldots, p_n, q_1, \ldots, q_m)$, the following are equivalent: 
\begin{enumerate}
	\item For all $f \in \mathcal{F}_m$, $A(p_1, \ldots, p_n, \top^{f(1)}, \ldots, \top^{f(m)})$ are contingent. 
	\item $\sFGH_T[A; q_1, \ldots, q_m]$ holds. 
\end{enumerate}
\end{thm}

Theorem \ref{Thm2} also follows from the following lemma that is an improvement of Lemma \ref{Lem1}. 
For the sake of simplicity, we sometimes abbreviate the tuples $p_1, \ldots, p_n$ and $q_1, \ldots, q_m$ as $\vec{p}$ and $\vec{q}$, respectively.  

\begin{lem}\label{Lem2}
Let $A(\vec{p}, \vec{q})$ be any propositional formula with only the indicated propositional variables and let $r$ be any propositional variable not contained in $A$. 
If $A(\vec{p}, \top^{f(1)}, \ldots, \top^{f(m)})$ is contingent for all $f \in \mathcal{F}_m$, then there exist propositional formulas $B_1(r, \vec{q}), \ldots, B_n(r, \vec{q})$ such that
\[
	\models r \leftrightarrow A(B_1(r, \vec{q}), \ldots, B_n(r, \vec{q}), \vec{q}).
\] 
\end{lem}
\begin{proof}
For each $f \in \mathcal{F}_m$, let $Q^f(\vec{q})$ denote the formula $q_1^{f(1)} \land \cdots \land q_m^{f(m)}$. 
Notice that for each distinct elements $f, g$ of $\mathcal{F}_m$, $\models \neg \bigl(Q^f(\vec{q}) \land Q^g(\vec{q}) \bigr)$. 
Also $\models \bigvee_{f \in \mathcal{F}_m} Q^f(\vec{q})$. 

Suppose that $A(\vec{p}, \top^{f(1)}, \ldots, \top^{f(m)})$ is contingent for all $f \in \mathcal{F}_m$. 
For each $g \in \mathcal{F}_m$, there exist propositional formulas $C_1^g(r), \ldots, C_n^g(r)$ such that 
\begin{equation}\label{equiv}
	\models r \leftrightarrow A(C_1^g(r), \ldots, C_n^g(r), \top^{g(1)}, \ldots, \top^{g(m)})
\end{equation}
by Lemma \ref{Lem1}. 
For each $i$ ($1 \leq i \leq n$), let $B_i(r, \vec{q})$ be the propositional formla $\bigvee_{f \in \mathcal{F}_m} \bigl(C_i^f(r) \land Q^f(\vec{q}) \bigr)$. 
Then, $\models Q^g(\vec{q}) \to \bigl(B_i(r, \vec{q}) \leftrightarrow C_i^g(r) \bigr)$ and $\models Q^g(\vec{q}) \to (q_i \leftrightarrow \top^{g(i)})$. 
By the equivalence (\ref{equiv}), we obtain 
\[
	\models Q^g(\vec{q}) \to \bigl(r \leftrightarrow A(B_1(r, \vec{q}), \ldots, B_n(r, \vec{q}), \vec{q}) \bigr). 
\]
Then, 
\[
	\models \bigvee_{f \in \mathcal{F}_m} Q^f(\vec{q}) \to \bigl(r \leftrightarrow A(B_1(r, \vec{q}), \ldots, B_n(r, \vec{q}), \vec{q}) \bigr). 
\]
Since $\models \bigvee_{f \in \mathcal{F}_m} Q^f(\vec{q})$, we conclude
\[
	\models r \leftrightarrow A(B_1(r, \vec{q}), \ldots, B_n(r, \vec{q}), \vec{q}). \tag*{\mbox{\qedhere}}
\]
\end{proof}

\begin{proof}[Proof of Theorem \ref{Thm2}]
$(1 \Rightarrow 2)$: 
Suppose that $A(\vec{p}, \top^{f(1)}, \ldots, \top^{f(m)})$ is contingent for all $f \in \mathcal{F}_m$. 
Let $\psi_1, \ldots, \psi_m$ be any $\LA$-sentences and let $\sigma$ be any $\Sigma_1$ sentence. 
By Lemma \ref{Lem2}, there exist propositional formulas $B_1(r, \vec{q}), \ldots, B_n(r, \vec{q})$ such that
\begin{equation}\label{equiv3}
	\models r \leftrightarrow A(B_1(r, \vec{q}), \ldots, B_n(r, \vec{q}), \vec{q}).
\end{equation}
By the FGH theorem, there exists an $\LA$-sentence $\varphi \in \FGH_T(\sigma)$. 
For each $i$ ($1 \leq i \leq n$), let $\varphi_i$ be the $\LA$-sentence $B_i(\varphi, \psi_1, \ldots, \psi_m)$. 
Then, by the equivalence (\ref{equiv3}), 
\[
	\PA \vdash \varphi \leftrightarrow A(\varphi_1, \ldots, \varphi_n, \psi_1, \ldots, \psi_m).
\]
By Proposition \ref{FGHeq}, we have $A(\varphi_1, \ldots, \varphi_n, \psi_1, \ldots, \psi_m) \in \FGH_T(\sigma)$. 
Therefore, we conclude that $\sFGH_T[A; \vec{q}]$ holds. 

$(2 \Rightarrow 1)$: Suppose that $A(p_1, \ldots, p_n, \top^{f(1)}, \ldots, \top^{f(m)})$ is either a tautology or unsatisfiable for some $f \in \mathcal{F}_m$. 
For $i \in \{1, \ldots, m\}$, let $\psi_i$ be the $\LA$-sentence $0 = 0^{f(i)}$. 
Then, for any $\LA$-sentences $\varphi_1, \ldots, \varphi_n$, the $\LA$-sentence $A(\varphi_1, \ldots, \varphi_n, \psi_1, \ldots, \psi_m)$ is either provable or refutable in $\PA$. 
Let $\sigma$ be any $\Sigma_1$ sentence independent of $\PA + \Con_T$. 
Then, $A(\varphi_1, \ldots, \varphi_n, \psi_1, \ldots, \psi_m) \notin \FGH_T(\sigma)$. 
Therefore $\sFGH_T[A; q_1, \ldots, q_m]$ does not hold. 
\end{proof}

For example, let $A(p, q)$ be the propositional formula $p \leftrightarrow q$. 
Since both $p \leftrightarrow \top$ and $p \leftrightarrow \bot$ are contingent, $\sFGH_T[A; q]$ holds by Theorem \ref{Thm2}. 
That is, for any $\LA$-sentence $\psi$ and any $\Sigma_1$ sentence $\sigma$, there exists an $\LA$-sentence $\varphi$ such that $\varphi \leftrightarrow \psi \in \FGH_T(\sigma)$. 
Another interesting corollary to Theorem \ref{Thm2} is the following version of the first incompleteness theorem. 

\begin{cor}\label{FI1}
Let $A(p_1, \ldots, p_n, q_1, \ldots, q_m)$ be any propositional formula. 
If $A(p_1, \ldots, p_n, \top^{f(1)}, \ldots, \top^{f(m)})$ are contingent for all $f \in \mathcal{F}_m$, then for any $\LA$-sentences $\psi_1, \ldots, \psi_m$, there exist $\LA$-sentences $\varphi_1, \ldots, \varphi_n$ such that $A(\varphi_1, \ldots, \varphi_n, \psi_1, \ldots, \psi_m)$ is independent of $T$. 
\end{cor}
\begin{proof}
Let $\sigma$ be a $\Sigma_1$ sentence independent of $\PA + \Con_T$ and $\psi_1, \ldots, \psi_m$ be any $\LA$-sentences. 
By Theorem \ref{Thm2}, there exist $\LA$-sentences $\varphi_1, \ldots, \varphi_n$ such that
\[
	\PA + \Con_T \vdash \sigma \leftrightarrow \PR_T(\gn{A(\varphi_1, \ldots, \varphi_n, \psi_1, \ldots, \psi_m)}). 
\]
Then, it is easy to show that $A(\varphi_1, \ldots, \varphi_n, \psi_1, \ldots, \psi_m)$ is independent of $T$. 

\end{proof}

Notice that we can also prove the parameterized version and the weak representability version of Theorem \ref{Thm2}.

\section{On the form of a sentence in the FGH theorem\\ -- the case of modal propositional logic}\label{Modal}

In this section, we extend Theorem \ref{Thm1} to the framework of modal propositional logic. 
This section consists of three subsections. 
First, we give some preparations which are needed for our arguments. 
In the second subsection, we prove our modal version of Theorem \ref{Thm1}. 
Finally, we also investigate our modal version for $\LA$-theories with finite heights. 

Here we define the height of theories. 
The sequence $\langle \Con_T^n \rangle_{n \in \omega}$ of $\LA$-sentences is recursively defined as follows: 
$\Con_T^0$ is $0 = 0$; and $\Con_T^{n+1}$ is $\neg \PR_T(\gn{\neg \Con_T^n})$. 
Notice that $\Con_T^1$ is $\PA$-provably equivalent to $\Con_T$. 
If there exists a natural number $k \geq 1$ such that $T \vdash \neg \Con_T^k$, then the least such a number $k$ is called the \textit{height} of $T$.  
If not, we say that the height of $T$ is $\infty$.

\subsection{Preparations}

The language of modal propositional logic is that of classical propositional logic equipped with the unary modal operation $\Box$. 
For each modal formula $A$, let $\Sub(A)$ be the set of all subformulas of $A$. 
Let $\cx(A)$ be the cardinality of the set $\{B \in \Sub(A) \mid B$ is of the form $\Box C\}$. 
The modal formula $\Box^n A$ is recursively defined as follows: $\Box^0 A$ is $A$; and $\Box^{n+1} A$ is $\Box \Box^n A$. 
The formula $\Diamond^n A$ is an abbreviation for $\neg \Box^n \neg A$.

The base logic of our investigations in this section is the G\"odel--L\"ob logic $\GL$ which is known as the logic of provability. 
The axioms of $\GL$ are as follows: 
\begin{enumerate}
	\item All propositional tautologies in the language of modal propositional logic, 
	\item $\Box(p \to q) \to (\Box p \to \Box q)$, 
	\item $\Box (\Box p \to p) \to \Box p$. 
\end{enumerate}
The inference rules of $\GL$ are modus ponens, necessitation, and uniform substitution. 

A \textit{$\GL$-frame} is a tuple $\langle W, \sqsubset, r \rangle$ where $W$ is a nonempty finite set, $\sqsubset$ is a transitive irreflexive binary relation on $W$ and $r$ is an element of $W$ with $r \sqsubset x$ for all $x \in W \setminus \{r\}$. 
Such an element $r$ is called the \textit{root} of the frame. 
A \textit{$\GL$-model} is a tuple $M = \langle W, \sqsubset, r, \Vdash \rangle$ where $\langle W, \sqsubset, r \rangle$ is a $\GL$-frame, 
and $\Vdash$ is a binary relation between $W$ and the set of all modal formulas satisfying the usual conditions for satisfaction and the following condition: 
$a \Vdash \Box A$ if and only if for all $b \in W$, $b \Vdash A$ if $a \sqsubset b$. 
It is known that $\GL$ is sound and complete with respect to the class of all $\GL$-frames (cf.~Segerberg \cite{Seg}). 
Moreover, the proof of the Kripke completeness of $\GL$ given in the textbook \cite{Boo} by Boolos shows that the following theorem holds. 

\begin{thm}\label{GLCompl}
For any modal formula $A$, if $\GL \nvdash A$, then there exists a $\GL$-model $\langle W, \sqsubset, r, \Vdash \rangle$ such that $r \Vdash \Box^{\cx(A) + 1} \bot$ and $r \nVdash A$.   
\end{thm}

For each set $\Gamma$ of modal formulas, let $\GL + \Gamma$ denote the logic whose axioms are all theorems of $\GL$ and all elements of $\Gamma$, and whose inference rules are modus ponens and uniform substitution. 
We identify each axiomatic system of modal propositional logic with the set of all its theorems. 
We introduce the following two extensions of $\GL$ which are studied in the context of the classification of propositional provability logics (cf.~Artemov and Beklemishev \cite{AB}). 

\begin{itemize}
	\item $\GL_\omega : = \GL + \{\Diamond^n \top \mid n \in \omega\}$. 
	\item $\GLS : = \GL + \{\Box p \to p\}$. 
\end{itemize}

Notice $\GL \subseteq \GL_\omega \subseteq \GLS$. 
To connect these logics with arithmetic, we introduce the notion of arithmetical interpretation. 

\begin{defn}[Arithmetical interpretations]
A mapping from the set of all propositional variables to a set of $\LA$-sentences is called an \textit{arithmetical interpretation}. 
Each arithmetical interpretation $f$ is uniquely extended to the mapping $f_T$ from the set of all modal formulas to a set of $\LA$-sentences by the following clauses: 
\begin{enumerate}
	\item $f_T(\bot)$ is $0=1$, 
	\item $f_T(\neg A)$ is $\neg f_T(A)$, 
	\item $f_T(A \circ B)$ is $f_T(A) \circ f_T(B)$ for $\circ \in \{\land, \lor, \to, \leftrightarrow\}$, 
	\item $f_T(\Box A)$ is $\PR_T(\gn{f_T(A)})$. 
\end{enumerate}
\end{defn}

The logics $\GL$, $\GL_\omega$ and $\GLS$ are sound with respect to arithmetical interpretations. 
Let $\N$ denote the standard model of arithmetic. 

\begin{fact}[Arithmetical soundness (cf.~Artemov and Beklemishev \cite{AB})]\label{AS}
Let $A$ be any modal formula. 
\begin{enumerate}
	\item If $\GL \vdash A$, then $\PA \vdash f_T(A)$ for any arithmetical interpretation $f$. 
	\item If the height of $T$ is $\infty$ and $\GL_\omega \vdash A$, then $\N \models f_T(A)$ for any arithmetical interpretation $f$. 
	\item If $T$ is sound and $\GLS \vdash A$, then $\N \models f_T(A)$ for any arithmetical interpretation $f$. 
\end{enumerate}
\end{fact}

We show some interrelationships between these logics. 
For each modal formula $A$, let $\Rf(A)$ denote the set $\{\Box B \to B \mid \Box B \in \Sub(A)\}$. 
Notice that the cardinality of the set $\Rf(A)$ is exactly $\cx(A)$. 	

\begin{fact}[Solovay \cite{Sol}]\label{GLvsGLS}
For any modal formula $A$, the following are equivalent: 
\begin{enumerate}
	\item $\GLS \vdash A$. 
	\item $\GL \vdash \bigwedge \Rf(A) \to A$. 
\end{enumerate}
\end{fact}

The following fact is a kind of folklore which is proved through arithmetical interpretations. 

\begin{fact}\label{LemBox}
For any modal formula $A$, the following are equivalent: 
\begin{enumerate}
	\item $\GL \vdash A$. 
	\item $\GL_\omega \vdash \Box A$. 
	\item $\GLS \vdash \Box A$. 
\end{enumerate}
\end{fact}

\begin{lem}\label{LemNegBox}
For any modal formula $A$, the following are equivalent: 
\begin{enumerate}
	\item $\GL \vdash \bigwedge \Rf(\Box A) \to \neg \Box A$. 
	\item $\GL \vdash \Diamond^{\cx(A)+1} \top \to \neg \Box A$. 
	\item $\GL_\omega \vdash \neg \Box A$. 
	\item $\GLS \vdash \neg \Box A$. 
\end{enumerate}
\end{lem}
\begin{proof}
$(1 \Rightarrow 2)$: 
Suppose $\GL \vdash \bigwedge \Rf(\Box A) \to \neg \Box A$. 
It is known that $\GL$ proves $\Diamond^{\cx(\Box A)} \top \to \bigwedge \Rf(\Box A) \lor \Diamond \bigwedge \Rf(\Box A)$ (cf.~\cite[Lemma 26]{AB}). 
Then, $\GL \vdash \Diamond^{\cx(\Box A)} \top \to \neg \Box A \lor \Diamond \neg \Box A$. 
Since $\GL \vdash \Diamond \neg \Box A \to \neg \Box A$, we have $\GL \vdash \Diamond^{\cx(\Box A)} \top \to \neg \Box A$. 
Hence, $\GL \vdash \Diamond^{\cx(A)+1} \top \to \neg \Box A$ because $\cx(\Box A) = \cx(A) + 1$. 

$(2 \Rightarrow 3)$ and $(3 \Rightarrow 4)$: Obvious. 

$(4 \Rightarrow 1)$: By Fact \ref{GLvsGLS} and $\Rf(\Box A) = \Rf(\neg \Box A)$. 
\end{proof}

The conditions ``$\GL \vdash A$'' and ``$\GL_{\omega} \vdash \neg \Box A$'' are also characterized by provability in other extensions of $\GL$.

\begin{defn}
Let $F_s$ be the modal formula $\Box^{s+1} \bot \to \Box^s \bot$. 
\end{defn}

\begin{lem}\label{LemBox2}
For any modal formula $A$ and any natural number $s$ with $s > \cx(A)$, the following are equivalent: 
\begin{enumerate}
	\item $\GL \vdash A$. 
	\item $\GL + \{\neg F_s\} \vdash \Box A$. 
\end{enumerate}
\end{lem}
\begin{proof}
$(1 \Rightarrow 2)$: Obvious. 

$(2 \Rightarrow 1)$: If $\GL \nvdash A$, then there exists a $\GL$-model $\langle W, \sqsubset, r, \Vdash \rangle$ such that $r \Vdash \Box^{\cx(A)+1} \bot \land \neg A$ by Theorem \ref{GLCompl}. 
Let $\langle W', \sqsubset', r', \Vdash' \rangle$ be a $\GL$-model obtained from $\langle W, \sqsubset, r, \Vdash \rangle$ by adding a chain of new elements below $r$ so that $r' \Vdash' \Box^{s+1} \bot \land \Diamond^s \top$. 
Such a model exists because $s > \cx(A)$. 
Since $r' \Vdash' \neg \Box A$, we obtain $\GL + \{\neg F_s\} \nvdash \Box A$. 
\end{proof}

\begin{lem}\label{LemNegBox2}
For any modal formula $A$ and any natural number $s$ with $s > \cx(A)$, the following are equivalent: 
\begin{enumerate}
	\item $\GL_\omega \vdash \neg \Box A$. 
	\item $\GL + \{\neg F_s\} \vdash \neg \Box A$. 
\end{enumerate}
\end{lem}
\begin{proof}
$(1 \Rightarrow 2)$: If $\GL_\omega \vdash \neg \Box A$, then by Lemma \ref{LemNegBox}, $\Diamond^{\cx(A)+1} \top \to \neg \Box A$ is proved in $\GL$, and thus $\GL + \{\neg F_s\} \vdash \neg \Box A$ because $s \geq \cx(A)+1$. 

$(2 \Rightarrow 1)$: 
Suppose $\GL + \{\neg F_s\} \vdash \neg \Box A$, then $\GL \vdash \Box A \to F_s$, that is, $\GL \vdash \Box A \to (\Box^{s+1} \bot \to \Box^s \bot)$. 
Then, $\GL \vdash \Box \Box A \to \Box(\Box^{s+1} \bot \to \Box^s \bot)$. 
By L\"ob's principle, we get $\GL \vdash \Box \Box A \to \Box^{s+1} \bot$. 
Then, $\GL \vdash \Box A \to \Box^{s+1} \bot$, and hence $\GL \vdash \Diamond^{s+1} \top \to \neg \Box A$. 
We obtain $\GL_\omega \vdash \neg \Box A$. 
\end{proof}

Here we introduce a notion that will be central in this section.

\begin{defn}
A modal formula $A$ is said to be \textit{nontrifling} if $\GL_\omega \nvdash \Box \Box A \to \Box A$. 
\end{defn}

Then, this notion is characterized in many ways as follows. 

\begin{prop}\label{NonT}
For any modal formula $A$ and any number $s > \cx(A)$, the following are equivalent: 
\begin{enumerate}
	\item $A$ is nontrifling. 
	\item $\GL_\omega \nvdash \Box A$ and $\GL_\omega \nvdash \neg \Box A$. 
	\item $\GLS \nvdash \Box A$ and $\GLS \nvdash \neg \Box A$. 
	\item $\GL + \{\neg F_s\} \nvdash \Box A$ and $\GL + \{\neg F_s\} \nvdash \neg \Box A$. 
	\item $\GL \nvdash A$ and $\GL \nvdash \bigwedge \Rf(\Box A) \to \neg \Box A$. 
\end{enumerate}
\end{prop}
\begin{proof}
$(1 \Rightarrow 2)$: 
We prove the contrapositive. 
If $\GL_\omega \vdash \Box A$, then $\GL_\omega \vdash \Box \Box A \to \Box A$ is obvious. 
Suppose $\GL_\omega \vdash \neg \Box A$. 
By Lemma \ref{LemNegBox}, we have $\GL \vdash \Diamond^{\cx(A)+1} \top \to \neg \Box A$. 
Then, $\GL \vdash \Diamond^{\cx(A)+2} \top \to \neg \Box \Box A$. 
We obtain $\GL_\omega \vdash \neg \Box \Box A$, and hence $\GL_\omega \vdash \Box \Box A \to \Box A$. 

$(2 \Leftrightarrow 3 \Leftrightarrow 4 \Leftrightarrow 5)$: By Fact \ref{LemBox} and Lemmas \ref{LemNegBox}, \ref{LemBox2}, and \ref{LemNegBox2}. 

$(5 \Rightarrow 1)$: Suppose $\GL \nvdash A$ and $\GL \nvdash \bigwedge \Rf(\Box A) \to \neg \Box A$. 
Then, $\GL \nvdash \Box A \to A$. 
By Theorem \ref{GLCompl}, there exist $\GL$-models $M = \langle W, \sqsubset, r, \Vdash \rangle$ and $M_0 = \langle W_0, \sqsubset_0, r_0, \Vdash_0 \rangle$ such that $r \Vdash \Box A \land \neg A$ and $r_0 \Vdash_0 \Rf(\Box A) \land \Box A$. 
We merge these two models $M$ and $M_0$ into one model $M^*$, which is illustrated in Fig \ref{Fig0}. 
We give the precise definition of the $\GL$-model $M^* = \langle W^*, \sqsubset^*, r^*, \Vdash^* \rangle$ as follows: 
\begin{itemize}
	\item $W^* = W \cup W_0 \cup \{r_i \mid i \geq 1\} \cup \{r^*\}$, 
	\item $\sqsubset^*$ is the transitive closure of 
	\[
		\sqsubset \cup \sqsubset_0 \cup \{(r_i, r_j) \mid i > j\} \cup \{(r^*, r_i) \mid i \in \omega\} \cup \{(r^*, r)\},
	\] 
	\item for $a \in W$, $a \Vdash^* p$ if and only if $a \Vdash p$, \\
	for $a \in W_0$, $a \Vdash^* p$ if and only if $a \Vdash_0 p$, and \\
	for $a \in \{r_i \mid i \geq 1\} \cup \{r^*\}$, $a \Vdash^* p$ if and only if $r_0 \Vdash_0 p$. 
\end{itemize}

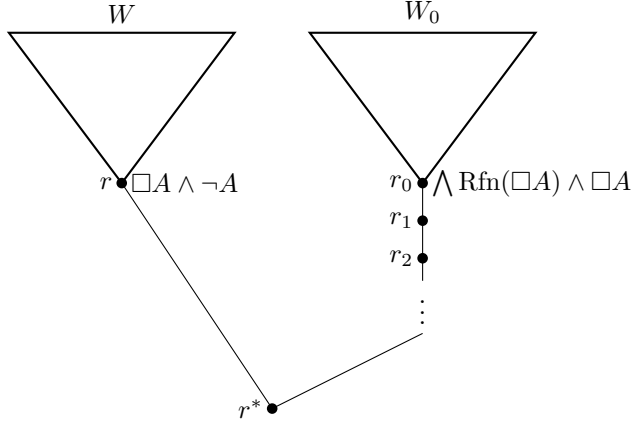
\begin{figure}[th]
\centering
\begin{tikzpicture}

\coordinate (root) at (0, -1); 
\coordinate (leftroot) at (-2, 2); 
\coordinate (rightroot0) at (2, 2); 
\coordinate (rightroot1) at (2, 1.5); 
\coordinate (rightroot2) at (2, 1); 
\coordinate (rightroot3) at (2, 0.7); 
\coordinate (right) at (2, 0); 

\fill (root) circle (2pt) (leftroot) circle (2pt) (rightroot0) circle (2pt) (rightroot1) circle (2pt) (rightroot2) circle (2pt) ;

\draw (root) -- (leftroot); 
\draw (root) -- (right); 
\draw (rightroot3) -- (rightroot2) -- (rightroot1) -- (rightroot0); 

\draw [thick] (leftroot) -- (-0.5, 4) -- (-3.5, 4) -- cycle; 
\draw [thick] (rightroot0) -- (0.5, 4) -- (3.5, 4) -- cycle;

\draw (root) node [left] {$r^*$};
\draw (leftroot) node [left] {$r$};
\draw (rightroot0) node [left] {$r_0$};
\draw (rightroot1) node [left] {$r_1$};
\draw (rightroot2) node [left] {$r_2$};
\draw (right) node [above] {$\vdots$};

\draw (leftroot) node [right] {$\Box A \land \neg A$};
\draw (rightroot0) node [right] {$\bigwedge \Rf(\Box A) \land \Box A$};

\draw (-2, 4) node [above] {$W$};
\draw (2, 4) node [above] {$W_0$};

\end{tikzpicture}
\caption{The $\GL$-model $M^*$}\label{Fig0}
\end{figure}

It is easy to show that $r \Vdash^* \Box A \land \neg A$ and $r_0 \Vdash^* \Rf(\Box A) \land \Box A$. 
Since $r^* \Vdash^* \Diamond^n \top$ for all $n \in \omega$, every theorem of $\GL_\omega$ is true in $r^*$. 
Thus, it suffices to show that $r^* \nVdash^* \Box \Box A \to \Box A$. 
For this purpose, we prove the following lemma. 

\begin{lem}\label{TLem}
For any $i \in \omega$ and any subformula $B$ of $\Box A$, $r_i \Vdash^* B$ if and only if $r_0 \Vdash^* B$. 
\end{lem}
\begin{proof}
We prove the lemma by induction on $i$. 
For $i = 0$, the lemma is trivial. 
Suppose that the lemma holds for $i$, and we prove the case of $i + 1$ by induction on the construction of the subformula $B$ of $\Box A$.  
The case that $B$ is a propositional variable is obvious from the definition of $\Vdash^*$. 
The cases for propositional connectives are easily proved by the induction hypothesis. 
We only give a proof of the case that $B$ is of the form $\Box C$, that is, we prove $r_{i+1} \Vdash^* \Box C$ if and only if $r_0 \Vdash^* \Box C$. 

$(\Rightarrow)$: 
Suppose $r_{i+1} \Vdash^* \Box C$, then $r_{i+1} \Vdash^* \Box \Box C$ because $\sqsubset^*$ is transitive. 
Since $r_{i+1} \sqsubset^* r_0$, we obtain $r_0 \Vdash^* \Box C$. 

$(\Leftarrow)$: 
Suppose $r_{i+1} \nVdash^* \Box C$. 
Then, $a \nVdash^* C$ for some $a \in W^*$ with $r_{i+1} \sqsubset^* a$. 
We distinguish the following three cases: 
\begin{itemize}
	\item $r_0 \sqsubset^* a$. \\
	Then, $r_0 \nVdash^* \Box C$ trivially holds. 
	\item $a = r_0$. \\
	We have $r_0 \Vdash^* \Box C \to C$ because $r_0 \Vdash^* \Rf(\Box A)$. 
	Thus, $r_0 \nVdash^* \Box C$.  
	\item 	$a = r_j$ for $0 < j < i+1$. \\
	By the induction hypothesis for $j$, we have $r_0 \nVdash^* C$. 
	Then, $r_j \nVdash^* \Box C$. 
	By the induction hypothesis for $j$ again, we obtain $r_0 \nVdash^* \Box C$. \qedhere
\end{itemize}
\end{proof}

We are ready to prove $r^* \nVdash^* \Box \Box A \to \Box A$. 
Let $a \in W^*$ be any element such that $r^* \sqsubset^* a$. 
If $a \in W$, then $a \Vdash^* \Box A$ because $r \Vdash^* \Box A$. 
If $a \in W_0$, then $a \Vdash^* \Box A$ because $r_0 \Vdash^* \Box A$. 
If $a \in \{r_i \mid i \geq 1\}$, then $a \Vdash^* \Box A$ by Lemma \ref{TLem}. 
Therefore, we obtain $r^* \Vdash^* \Box \Box A$. 
Since $r^* \sqsubset^* r$ and $r \nVdash^* A$, we get $r^* \nVdash^* \Box A$. 
Thus, we conclude $r^* \nVdash^* \Box \Box A \to \Box A$, and hence $A$ is nontrifling. 
\end{proof}

If a formula $A$ does not contain $\Box$, then the following proposition holds. 

\begin{prop}\label{ContNonT}
A propositional formula $A$ is contingent if and only if $A$ is nontrifling. 
\end{prop}
\begin{proof}
$(\Rightarrow)$: 
Let $\mathsf{v}_0$ be a truth assignment in classical propositional logic such that $\mathsf{v}_0(A)$ is false. 
Let $\langle W_0, \sqsubset_0, r_0, \Vdash_0 \rangle$ be a $\GL$-model satisfying $W_0 = \{r_0\}$, $\sqsubset_0 = \varnothing$, and $r_0 \Vdash_0 p$ if and only if $\mathsf{v}_0(p)$ is true. 
Then $r_0 \nVdash_0 A$, and hence $\GL \nvdash A$. 

Let $\mathsf{v}_1$ be a truth assignment such that $\mathsf{v}_1(A)$ is true. 
Let $\langle W_1, \sqsubset_1, r_1, \Vdash_1 \rangle$ be a $\GL$-model satisfying $W_1 = \{r_1, a\}$, $r_1 \sqsubset_1 a$, and for each $w \in W_1$, $w \Vdash_1 p$ if and only if $\mathsf{v}_1(p)$ is true. 
Then, $r_1 \Vdash_1 \Diamond \top \land \Box A$. 
Since $\cx(A) = 0$, we have $\GL \nvdash \Diamond^{\cx(A)+1} \top \to \neg \Box A$. 
By Lemma \ref{LemNegBox}, $\GL_\omega \nvdash \neg \Box A$. 
By Proposition \ref{NonT}, $A$ is nontrifling. 

$(\Leftarrow)$: 
If $A$ is a tautology, then $\GL \vdash A$. 
If $A$ is unsatisfiable, then $\GL \vdash \neg A$, and $\GL \vdash \Diamond \top \to \neg \Box A$. 
Hence, $\GL_\omega \vdash \neg \Box A$. 
In either case, $A$ is not nontrifling by Proposition \ref{NonT}. 
\end{proof}

\subsection{A modal version of the FGH theorem}

We are ready to prove an extension of Theorem \ref{Thm1}. 

\begin{thm}\label{MT}
For any modal formula $A$, if $A$ is nontrifling, then for any $\Sigma_1$ sentence $\sigma$, there exists an arithmetical interpretation $f$ such that 
\begin{description}
	\item [(a)] $\PA \vdash \sigma \to f_T(\Box A)$, and 
	\item [(b)] $\PA + \Con_T^{\cx(A) + 1} \vdash f_T(\Box A) \to \sigma$. 
\end{description}
\end{thm}
\begin{proof}
Suppose that $A$ is nontrifling and let $\sigma$ be any $\Sigma_1$ sentence. 
We may assume that $\sigma$ is of the form $\exists x \, \delta(x)$ for some $\Delta_0$ formula $\delta(x)$. 
By Proposition \ref{NonT}, we have $\GL \nvdash A$ and $\GL \nvdash \bigwedge \Rf(\Box A) \to \neg \Box A$. 
By Theorem \ref{GLCompl}, there exist $\GL$-models $M_0 = \pair{W_0, \sqsubset_0, r_0, \Vdash_0}$ and $M_1 = \pair{W_1, \sqsubset_1, r_1, \Vdash_1}$ such that $r_0 \Vdash_0 \Box^{\cx(A)+1} \bot \land \neg A$ and $r_1 \Vdash_1 \bigwedge \Rf(\Box A) \land \Box A$.

Let $\langle \cdot, \cdot \rangle$ be a natural injective mapping from $\omega^2$ to $\omega$. 
Also let $\pi_0$ and $\pi_1$ be natural mappings such that $\pi_0(\pair{i, j}) = i$ and $\pi_1(\pair{i, j}) = j$. 
We may assume that $0$ is not in the range of the mapping $\langle \cdot, \cdot \rangle$ and is not in the domain of the mappings $\pi_0$ and $\pi_1$. 
Let $n_0$ and $n_1$ be the cardinalities of $W_0$ and $W_1$, respectively. 
Without loss of generality, we may assume 
\begin{itemize}
	\item $W_0 = \{\pair{0, 1}, \pair{0, 2}, \ldots, \pair{0, n_0}\}$, $r_0 = \pair{0, 1}$, and 
	\item $W_1 = \{\pair{1, 1}, \pair{1, 2}, \ldots, \pair{1, n_1}\}$, $r_1 = \pair{1, 1}$. 
\end{itemize}
We merge the two models $M_0$ and $M_1$ into one $\GL$-model $M$, which is visualized in Figure \ref{Fig1}. 
The definition of $M = \pair{W', \sqsubset', 0, \Vdash'}$ is as follows: 
\begin{itemize}
	\item $W' = W_0 \cup W_1 \cup \{0, \pair{0, 0}, \pair{1, 0}\}$, 
	\item $\sqsubset' = \sqsubset_0 \cup \sqsubset_1 \cup \{(0, a) \mid a \in W' \setminus \{0\}\} \cup \{(\pair{i, 0}, a) \mid i \in \{0, 1\}, a \in W_i\}$, 
	\item For $\pair{i, j} \in W_i$, $\pair{i, j} \Vdash' p$ if and only if $\pair{i, j} \Vdash_i p$. \\
	Also $0 \Vdash' p$, and $\pair{i, 0} \Vdash' p$ if and only if $\pair{i, 1} \Vdash_i p$. 
\end{itemize}

\begin{figure}[th]
\centering
\begin{tikzpicture}

\coordinate (root) at (0, 0); 
\coordinate (leftroot0) at (-2, 1); 
\coordinate (rightroot0) at (2, 1); 
\coordinate (leftroot1) at (-2, 2); 
\coordinate (rightroot1) at (2, 2); 

\fill (root) circle (2pt) (leftroot0) circle (2pt) (rightroot0) circle (2pt) (leftroot1) circle (2pt) (rightroot1) circle (2pt);

\draw (root) -- (leftroot0) -- (leftroot1); 
\draw (root) -- (rightroot0) -- (rightroot1); 

\draw [thick] (leftroot1) -- (-0.5, 4) -- (-3.5, 4) -- cycle; 
\draw [thick] (rightroot1) -- (0.5, 4) -- (3.5, 4) -- cycle;

\draw (root) node [left] {$0$};
\draw (leftroot0) node [left] {$\pair{0, 0}$};
\draw (leftroot1) node [left] {$\pair{0, 1}$};
\draw (rightroot0) node [left] {$\pair{1, 0}$};
\draw (rightroot1) node [left] {$\pair{1, 1}$};

\draw (leftroot1) node [right] {$\Box^{\cx(A)+1} \bot \land \neg A$};
\draw (rightroot1) node [right] {$\bigwedge \Rf(\neg \Box A) \land \Box A$};

\draw (-2, 4) node [above] {$W_0$};
\draw (2, 4) node [above] {$W_1$};

\end{tikzpicture}
\caption{The $\GL$-model $M'$}\label{Fig1}
\end{figure}
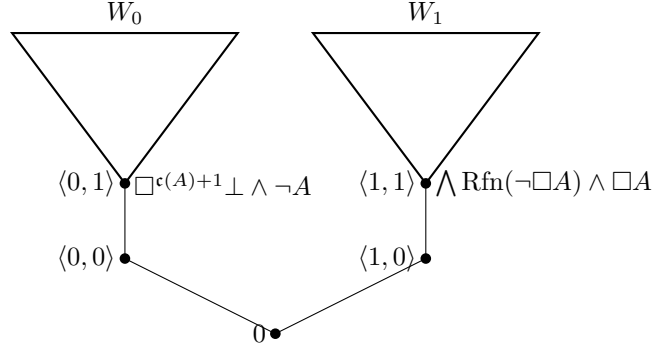

Then, it is easy to show that $\pair{0,1} \Vdash' \Box^{\cx(A) + 1} \bot \land \neg A$ and $\pair{1,1} \Vdash' \bigwedge \Rf(\neg \Box A) \land \Box A$. 
As in the usual proof of Solovay's arithmetical completeness theorem, we recursively define a primitive recursive function $h : \omega \to W'$ by referring to $T$-proofs. 
In the definition of $h$, we use the formula $\lambda(x) \equiv \exists y \, \forall z \, {\geq} y \, (h(z) = x)$ and the sentence $f_T(A)$ for the arithmetical interpretation $f$ defined by $f(p) \equiv \bigvee_{\substack{a \in W' \\ a \Vdash' p}} \lambda(\num{a})$. 
This is done by the aid of the Fixed Point Lemma or the recursion theorem as in the proof of Solovay's theorem because such $\lambda(x)$ and $f_T(A)$ are effectively computable from $h$ (see Boolos \cite{Boo}). 

Here we give the definition of the function $h$. 
Let $h(0) : = 0$. 
Suppose that the value of $h(x)$ has already been defined. 
We define the value of $h(x+1)$. 
If $\forall y \, {\leq}\, x \, \neg \delta(y)$ holds and there exists no $T$-proof of $f_T(A)$ smaller than or equal to $x$, let $h(x+1) : = 0$. 
That is, the output of $h$ remains $0$ unless the smallest witness of $\delta(x)$ or the smallest $T$-proof of $f_T(A)$ appears. 

After such a witness appears, $h$ changes its value.
At the first stage when the smallest witness $x$ of $\delta(x)$ or the smallest $T$-proof $x$ of $f_T(A)$ appears, we distinguish the following two cases:
\begin{itemize}
	\item Case 1: $\forall y \, {<} \, x \, \neg \delta(y)$ holds and $x$ is the smallest $T$-proof of $f_T(A)$. \\
	Let $h(x+1) : = \pair{0, 0}$. 
	
	\item Case 2: $\forall y \, {<} \, x \, \neg \delta(y)$ and $\delta(x)$ hold and there is no proof of $f_T(A)$ less than or equal to $x$. \\
	Let $h(x+1) : = \pair{1, 0}$. 

\end{itemize}
That is, Cases 1 and 2 correspond to the situations $\PR_T(\gn{f_T(A)}) \pc \sigma$ and $\sigma \prec \PR_T(\gn{f_T(A)})$, respectively. 

After that, we define the value of $h(x+1)$ as follows: 
\[
	h(x+1) : = \begin{cases} a & \text{if}\ h(x) \sqsubset' a\ \&\ x\ \text{is a}\ T\text{-proof of}\ \neg \lambda(\num{a}); \\ h(x) & \text{otherwise.} \end{cases}
\]

The definition of $h$ is hereby completed. 

We introduce three lemmas. 
The first lemma concerns the general properties of the function $h(x)$, the formula $\lambda(x)$, and the arithmetical interpretation $f_T$. 

\begin{lem}\label{L1}
Let $a, b \in W'$. 
\begin{enumerate}
	\item If $a \neq b$, then $\PA \vdash \lambda(\num{a}) \to \neg \lambda(\num{b})$, 
	\item $\displaystyle \PA \vdash h(x) = \num{a} \to \lambda(\num{a}) \lor \bigvee_{a \sqsubset' c} \lambda(\num{c})$, 
	\item $\displaystyle \PA \vdash \PR_T(\gn{f_T(A)}) \pc \sigma \leftrightarrow \bigvee_{\pi_0(c) = 0} \lambda(\num{c})$, 
	\item $\displaystyle \PA \vdash \sigma \prec \PR_T(\gn{f_T(A)}) \leftrightarrow \bigvee_{\pi_0(c) = 1} \lambda(\num{c})$, 
	\item If $a \sqsubset' b$, then $\PA \vdash \lambda(\num{a}) \to \neg \PR_T(\gn{\neg \lambda(\num{b})})$, 
	\item If $\pi_1(a) \geq 1$, then $\PA \vdash \lambda(\num{a}) \to \PR_T(\gn{\neg \lambda(\num{a})})$, 
	\item If $\pi_1(a) \geq 1$, then $\PA \vdash \lambda(\num{a}) \to \PR_T(\gn{\bigvee_{a \sqsubset' c} \lambda(\num{c})})$. 
\end{enumerate}
\end{lem}
\begin{proof}
Except the implications $\leftarrow$ in Clauses 3 and 4, these statements are proved in a similar way as in the usual proof of Solovay's arithmetical completeness theorem (cf.~\cite{AB,Boo,JD,Sol}). 

For the implication $\leftarrow$ in Clause 3, we argue in $\PA$: 
Suppose that $\lambda(\num{c})$ holds and $\pi_0(c) = 0$. 
Then, $h(k) = c$ for some $k$. 
If $\neg \sigma$ and $\neg \PR_T(\gn{f_T(A)})$ hold, then $\forall x \, h(x) = 0$ holds, a contradiction. 
Hence, $\PR_T(\gn{f_T(A)}) \pc \sigma$ or $\sigma \prec \PR_T(\gn{f_T(A)})$ holds. 
By Clause 1, we have $\neg \bigvee_{\pi_0(d) = 1} \lambda(\num{d})$. 
Then, by the implication $\to$ in Clause 4, $\sigma \prec \PR_T(\gn{f_T(A)})$ does not hold. 
Therefore, $\PR_T(\gn{f_T(A)}) \pc \sigma$ holds. 

The implication $\leftarrow$ in Clause 4 is also proved similarly. 
\end{proof}

The second lemma states that the satisfaction relation for $a \in W'$ with $\pi_1(a) \geq 1$ is embedded into $\PA$. 

\begin{lem}\label{L2}
Let $a \in W'$ with $\pi_1(a) \geq 1$ and let $B$ be a modal formula. 
\begin{enumerate}
	\item If $a \Vdash' B$, then $\PA \vdash \lambda(\num{a}) \to f_T(B)$.  
	\item If $a \nVdash' B$, then $\PA \vdash \lambda(\num{a}) \to \neg f_T(B)$. 
\end{enumerate}
\end{lem}
\begin{proof}
Clauses 1 and 2 are proved by induction on the construction of $B$ simultaneously as in the usual proof of Solovay's theorem. 
\end{proof}

The third lemma concerns the satisfaction relation for the element $\pair{1,1}$. 

\begin{lem}\label{L3}
Let $B$ be any subformula of $\Box A$. 
\begin{enumerate}
	\item If $\pair{1, 1} \Vdash' B$, then $\PA \vdash \lambda(\lp{1, 0}) \to f_T(B)$,  
	\item If $B$ is $\Box C$ and $\pair{1, 1} \Vdash' \Box C$, then $\PA \vdash \sigma \prec \PR_T(\gn{f_T(A)}) \to f_T(\Box C)$,  
	\item If $\pair{1, 1} \nVdash' B$, then $\PA \vdash \lambda(\lp{1, 0}) \to \neg f_T(B)$.  
\end{enumerate}
\end{lem}
\begin{proof}
We prove Clauses 1, 2 and 3 by induction on the construction of $B \in \Sub(\Box A)$ simultaneously. 
We only give a proof of the case that $B$ is $\Box C$. 

1 and 2: Suppose $\pair{1, 1} \Vdash' \Box C$. 
Since $\Box C \in \Sub(\Box A)$, we have $\pair{1, 1} \Vdash' \Box C \to C$ because $\pair{1, 1} \Vdash' \bigwedge \Rf(\Box A)$. 
Then, $\pair{1, 1} \Vdash' C$. 
Hence, for all $a \in W'$ with $\pi_0(a) = 1$ and $\pi_1(a) \geq 1$, we have $a \Vdash' C$. 
By Lemma \ref{L2}.1, $\PA \vdash \bigvee_{\substack{\pi_0(a) = 1 \\ \pi_1(a) \geq 1}} \lambda(\num{a}) \to f_T(C)$. 
Also, by the induction hypothesis, $\PA \vdash \lambda(\lp{1, 0}) \to f_T(C)$. 
Thus, 
\[
	\PA \vdash \bigvee_{\pi_0(a) = 1} \lambda(\num{a}) \to f_T(C). 
\]
By Lemma \ref{L1}.4, 
\[
	\PA \vdash \sigma \prec \PR_T(\gn{f_T(A)}) \to f_T(C).
\]
Then,
\[
	\PA \vdash \PR_T(\gn{\sigma \prec \PR_T(\gn{f_T(A)})}) \to f_T(\Box C).
\] 
Since $\sigma \prec \PR_T(\gn{f_T(A)})$ is a $\Sigma_1$ sentence, we have
\[
	\PA \vdash \sigma \prec \PR_T(\gn{f_T(A)}) \to f_T(\Box C). 
\]
This completes the proof of Clause 2. 
$\PA \vdash \lambda(\lp{1,0}) \to \sigma \prec \PR_T(\gn{f_T(A)})$ by Lemma \ref{L1}.4, and hence Clause 1 also holds. 

3: Suppose $\pair{1, 1} \nVdash' \Box C$. 
Then, $a \nVdash' C$ for some $a \in W'$ with $\pair{1, 1} \sqsubset' a$. 
By Lemma \ref{L2}.2, we have 
\[
	\PA \vdash \lambda(\num{a}) \to \neg f_T(C).
\]
Then, 
\[
	\PA \vdash \neg \PR_T(\gn{\neg \lambda(\num{a})}) \to \neg f_T(\Box C).
\] 
Since $\pair{1, 0} \sqsubset' a$, by Lemma \ref{L1}.5, 
\[
	\PA \vdash \lambda(\lp{1, 0}) \to \neg \PR_T(\gn{\neg \lambda(\num{a})}). 
\]
Therefore, $\PA \vdash \lambda(\lp{1, 0}) \to \neg f_T(\Box C)$. 
\end{proof}

We are ready to show the required two statements: (a) $\PA \vdash \sigma \to f_T(\Box A)$ and (b) $\PA + \Con_T^{\cx(A)+1} \vdash f_T(\Box A) \to \sigma$. 

(a): Since $\pair{1, 1} \Vdash' \Box A$, by Lemma \ref{L3}.2, 
\[
	\PA \vdash \sigma \prec \PR_T(\gn{f_T(A)}) \to f_T(\Box A). 
\]
Since
\[
	\PA \vdash \sigma \land \neg f_T(\Box A) \to \sigma \prec \PR_T(\gn{f_T(A)}), 
\]
we have
\[
	\PA \vdash \sigma \land \neg f_T(\Box A) \to f_T(\Box A), 
\]
and hence we obtain $\PA \vdash \sigma \to f_T(\Box A)$. 

(b): Since $\pair{0, 1} \Vdash' \Box^{\cx(A)+1} \bot$, for all $a \in W'$ with $\pi_0(a) = 0$ and $\pi_1(a) \geq 1$, we have $a \Vdash' \Box^{\cx(A)+1} \bot$. 
By Lemma \ref{L2}.1, 
\[
	\PA \vdash \bigvee_{\substack{\pi_0(a)=0 \\ \pi_1(a) \geq 1}} \lambda(\num{a}) \to \neg \Con_T^{\cx(A)+1},
\]
and thus 
\begin{align}\label{fml1}
	\PA + \Con_T^{\cx(A)+1} \vdash \bigwedge_{\substack{\pi_0(a)=0 \\ \pi_1(a) \geq 1}} \neg \lambda(\num{a}). 
\end{align}
Since 
\[
	\PA \vdash f_T(\Box A) \land \neg \sigma \to \PR_T(\gn{f_T(A)}) \pc \sigma, 
\]
by Lemma \ref{L1}.3, we have
\[
	\PA \vdash f_T(\Box A) \land \neg \sigma \to \bigvee_{\pi_0(a) = 0} \lambda(\num{a}). 
\]
By combining this with (\ref{fml1}), we obtain
\begin{align}\label{fml2}
	\PA + \Con_T^{\cx(A)+1} \vdash f_T(\Box A) \land \neg \sigma \to \lambda(\lp{0, 0}). 
\end{align}

Since $\pair{0, 1} \nVdash' A$, by Lemma \ref{L2}.2, $\PA \vdash \lambda(\lp{0, 1}) \to \neg f_T(A)$. 
Then, 
\[
	\PA \vdash \neg \PR_T(\gn{\neg \lambda(\lp{0, 1})}) \to \neg f_T(\Box A). 
\]
Since $\pair{0, 0} \sqsubset' \pair{0,1}$, by Lemma \ref{L1}.5, 
\[
	\PA \vdash \lambda(\lp{0, 0}) \to \neg \PR_T(\gn{\neg \lambda(\lp{0, 1})}), 
\]
and hence we have $\PA \vdash \lambda(\lp{0, 0}) \to \neg f_T(\Box A)$. 
From this with (\ref{fml2}), we obtain $\PA + \Con_T^{\cx(A)+1} \vdash f_T(\Box A) \land \neg \sigma \to \neg f_T(\Box A)$. 
We conclude $\PA + \Con_T^{\cx(A)+1} \vdash f_T(\Box A) \to \sigma$. 
\end{proof}

If the height of $T$ is larger than $\cx(A)$, then the converse of Theorem \ref{MT} also holds. 

\begin{prop}\label{CMT}
Let $A$ be any modal formula. 
Suppose that the height of $T$ is larger than $\cx(A)$ and, for any $\Sigma_1$ sentence $\sigma$, there exists an arithmetical interpretation $f$ such that $\PA \vdash \sigma \to f_T(\Box A)$ and $\PA + \Con_T^{\cx(A) + 1} \vdash f_T(\Box A) \to \sigma$. 
Then, $A$ is nontrifling. 
\end{prop}
\begin{proof}
Since $0 = 1$ is a $\Sigma_1$ sentence, there exists an arithmetical interpretation $f$ such that $\PA + \Con_T^{\cx(A) + 1} \vdash f_T(\Box A) \to 0 = 1$. 
Since $\N \models \Con_T^{\cx(A) + 1}$, we have $\N \models \neg f_T(\Box A)$. 
Thus, $T \nvdash f_T(A)$. 
By Fact \ref{AS}.1, $\GL \nvdash A$. 

Also, since $0 = 0$ is a $\Sigma_1$ sentence, there exists an arithmetical interpretation $g$ such that $\PA \vdash 0 = 0 \to g_T(\Box A)$. 
Then, $\PA \vdash g_T(\Box A)$. 
Since $\N \models \Con_T^{\cx(A)+1}$, $\PA \nvdash \neg \Con_T^{\cx(A)+1}$, and thus $\PA \nvdash \Con_T^{\cx(A)+1} \to \neg g_T(\Box A)$. 
By Fact \ref{AS}.1, $\GL \nvdash \Diamond^{\cx(A)+1} \top \to \neg \Box A$. 
By Lemma \ref{LemNegBox}, $\GL_\omega \nvdash \neg \Box A$. 
Then, by Proposition \ref{NonT}, $A$ is nontrifling. 
\end{proof}

For any classical propositional formula $A$, we have $\cx(A) = 0$. 
So Proposition \ref{ContNonT} shows that Theorem \ref{MT} is actually an extension of Theorem \ref{Thm1}. 

For a variable $v$ of first-order logic, an \textit{$v$-arithmetical interpretation} $f$ is a mapping where for each propositional variable $p$, $f(p)$ is an $\LA$-formula $\varphi(v)$ with only the free variable $v$. 
Each $v$-arithmetical interpretation $f$ is uniquely extended to the mapping $f_T$ from the set of all modal formulas to a set of $\LA$-formulas with at most the free variable $v$ as the usual arithmetical interpretations with the clause $f_T(\Box A)(v) \equiv \PR_T(\gn{f_T(A)(\dot{v})})$. 
By tracing the proof of Theorem \ref{MT} entirely using the function $h(x, v)$ and the formula $\lambda(x, v)$, the following parameterized version of Theorem \ref{MT} is also proved.  

\begin{thm}\label{MT2}
For any modal formula $A$, if $A$ is nontrifling, then for any $\Sigma_1$ formula $\sigma(v)$, there exists an $v$-arithmetical interpretation $f$ such that 
\begin{enumerate}
	\item $\PA \vdash \forall v\, \bigl(\sigma(v) \to f_T(\Box A)(v) \bigr)$, and 
	\item $\PA + \Con_T^{\cx(A) + 1} \vdash \forall v\, \bigl(f_T(\Box A)(v) \to \sigma(v) \bigr)$. 
\end{enumerate}
\end{thm}

From Theorem \ref{MT2}, we obtain an extension of Theorem \ref{Thm1'} to the framework of modal logic. 

\begin{thm}\label{MT3}
Let $A$ be any modal formula $A$. 
If $A$ is nontrifling and the height of $T$ is larger than $\cx(A)$, then for any c.e.~set $X$, there exists an $v$-arithmetical interpretation $f$ such that $f_T(A)(v)$ weakly represents $X$ in $T$. 
\end{thm}
\begin{proof}
Let $\sigma(v)$ be a $\Sigma_1$ formula defining $X$ over $\N$. 
By Theorem \ref{MT2}, there exists an $v$-arithmetical interpretation $f$ such that 
\begin{enumerate}
	\item $\PA \vdash \forall v \, \bigl(\sigma(v) \to f_T(\Box A)(v) \bigr)$, and 
	\item $\PA + \Con_T^{\cx(A) + 1} \vdash \forall v\, \bigl(f_T(\Box A)(v) \to \sigma(v) \bigr)$. 
\end{enumerate}
Since $\N \models \Con_T^{\cx(A)+1}$, we have $\N \models \forall v\, \bigl(\sigma(v) \leftrightarrow f_T(\Box A)(v) \bigr)$. 
Then, for any $k \in \omega$, $k \in X$ if and only if $T \vdash f_T(A)(\num{k})$. 
This means that $f_T(A)(v)$ weakly represents $X$ in $T$.  
\end{proof}

\subsection{A modal version of the FGH theorem for theories with finite heights}

Notice that Theorem \ref{MT3} requires the assumption that the height of $T$ is larger than $\cx(A)$. 
In this subsection, we investigate variations of Theorems \ref{MT} and \ref{MT3}, keeping in mind theories with finite heights. 

It is already proved in Proposition \ref{NonT} that for any natural number $s$ with $s > \cx(A)$, $A$ is nontrifling if and only if $\GL + \{\neg F_s\} \nvdash \Box A$ and $\GL + \{\neg F_s\} \nvdash \neg \Box A$. 
On the other hand, even for a natural number $s$ with $s \leq \cx(A)$, we prove the following version of the FGH theorem concerning the condition ``$\GL + \{\neg F_s\} \nvdash \Box A$ and $\GL + \{\neg F_s\} \nvdash \neg \Box A$''.

\begin{thm}\label{MT4}
For any modal formula $A$, if $\GL + \{\neg F_s\} \nvdash \Box A$ and $\GL + \{\neg F_s\} \nvdash \neg \Box A$, then for any $\Sigma_1$ sentence $\sigma$, there exists an arithmetical interpretation $f$ such that 
\[
	\PA + \neg \Con_T^{s+1} + \Con_T^s \vdash \sigma \leftrightarrow f_T(\Box A). 
\]
\end{thm}
\begin{proof}
Suppose $\GL + \{\neg F_s\} \nvdash \Box A$ and $\GL + \{\neg F_s\} \nvdash \neg \Box A$, that is, $\GL \nvdash \Box^{s+1} \bot \land \Diamond^s \top \to \Box A$ and $\GL \nvdash \Box^{s+1} \bot \land \Diamond^s \top \to \neg \Box A$. 
By the Kripke completeness of $\GL$, there exist $\GL$-models $M_0 = \pair{W_0, \sqsubset_0, r_0, \Vdash_0}$ and $M_1 = \pair{W_1, \sqsubset_1, r_1, \Vdash_1}$ such that $r_0 \Vdash_0 \Box^{s+1} \bot \land \Diamond^s \top \land \neg \Box A$ and $r_1 \Vdash_1 \Box^{s+1} \bot \land \Diamond^s \top \land \Box A$. 
Let $n_0$ and $n_1$ be the cardinalities of $W_0$ and $W_1$ respectively, and we may assume that 
\begin{itemize}
	\item $W_0 = \{\pair{0, 1}, \pair{0, 2}, \ldots, \pair{0, n_0}\}$, $r_0 = \pair{0, 1}$, and
	\item $W_1 = \{\pair{1, 1}, \pair{1, 2}, \ldots, \pair{1, n_1}\}$, $r_1 = \pair{1, 1}$. 
\end{itemize}
Let $M' = \langle W', \sqsubset', 0, \Vdash' \rangle$ be the $\GL$-model obtained from $M_0$ and $M_1$ as in the proof of Theorem \ref{MT}. 
Then, $\pair{0,1} \Vdash' \Box^{s+1} \bot \land \Diamond^s \top \land \neg \Box A$ and $\pair{1,1} \Vdash' \Box^{s+1} \bot \land \Diamond^s \top \land \Box A$. 

We can define a function $h$, a formula $\lambda(x)$, and an arithmetical interpretation $f$ from the $\GL$-model $M'$ in the same way as in the proof of Theorem \ref{MT}. 
Then, the same results as Lemmas \ref{L1}, \ref{L2} and \ref{L3} can also be proved. 

For each $a \in W'$ with $\pi_1(a) > 1$, since $a \Vdash' \Box^s \bot$, we have that $\PA$ proves $\lambda(\num{a}) \to \neg \Con_T^s$, and thus 
\[
	\PA + \Con_T^s \vdash \bigwedge_{\pi_1(a) > 1} \neg \lambda(\num{a}). 
\]
Also, since $\pair{i, 1} \Vdash' \Diamond^s \top$, we have $\PA \vdash \lambda(\lp{i, 1}) \to \Con_T^s$. 
Then, $\PA$ proves $\neg \PR_T(\gn{\neg \lambda(\lp{i, 1})}) \to \Con_T^{s+1}$. 
Also, $\PA \vdash \lambda(\lp{i, 0}) \to \neg \PR_T(\gn{\neg \lambda(\lp{i, 1})})$ because $\pair{i, 0} \sqsubset' \pair{i, 1}$, and hence $\PA \vdash \lambda(\lp{i, 0}) \to \Con_T^{s+1}$. 
Therefore, we obtain 
\begin{align}\label{fml3}
	\PA + \neg \Con_T^{s+1} + \Con_T^s \vdash \bigwedge_{\pi_1(a) \neq 1} \neg \lambda(\num{a}). 
\end{align}

We prove $\PA + \neg \Con_T^{s+1} + \Con_T^s \vdash \sigma \leftrightarrow f_T(\Box A)$. 

$(\rightarrow)$: 
Since $\PA \vdash \sigma \prec \PR_T(\gn{f_T(A)}) \to \bigvee_{\pi_0(a) = 1} \lambda(\num{a})$, from (\ref{fml3}), we have 
	\[
		\PA + \neg \Con_T^{s+1} + \Con_T^s \vdash \sigma \prec \PR_T(\gn{f_T(A)}) \to \lambda(\lp{1, 1}).
	\] 
	Since $\pair{1, 1} \Vdash' \Box A$, we have $\PA \vdash \lambda(\lp{1, 1}) \to f_T(\Box A)$. 
Hence,
	\[
		\PA + \neg \Con_T^{s+1} + \Con_T^s \vdash \sigma \prec \PR_T(\gn{f_T(A)}) \to f_T(\Box A).
	\] 
	Here $\PA \vdash \sigma \land \neg f_T(\Box A) \to \sigma \prec \PR_T(\gn{f_T(A)})$, and thus
	\[
		\PA + \neg \Con_T^{s+1} + \Con_T^s \vdash \sigma \land \neg f_T(\Box A) \to f_T(\Box A).
	\] 
We obtain $\PA + \neg \Con_T^{s+1} + \Con_T^s \vdash \sigma \to f_T(\Box A)$. 

$(\leftarrow)$: 
Since $\PA \vdash \PR_T(\gn{f_T(A)}) \pc \sigma \to \bigvee_{\pi_0(a) = 0} \lambda(\num{a})$, from (\ref{fml3}), 
	\[
		\PA + \neg \Con_T^{s+1} + \Con_T^s \vdash \PR_T(\gn{f_T(A)}) \pc \sigma \to \lambda(\lp{0, 1}).
	\] 
Since $\pair{0, 1} \nVdash' \Box A$, $\PA \vdash \lambda(\lp{0, 1}) \to \neg f_T(\Box A)$. 
Therefore, 
	\[
		\PA + \neg \Con_T^{s+1} + \Con_T^s \vdash \PR_T(\gn{f_T(A)}) \pc \sigma \to  \neg f_T(\Box A), 
	\] 
and hence 
	\[
		\PA + \neg \Con_T^{s+1} + \Con_T^s \vdash f_T(\Box A) \land \neg \sigma \to \neg f_T(\Box A).
	\] 
We conclude $\PA + \neg \Con_T^{s+1} + \Con_T^s \vdash f_T(\Box A) \to \sigma$. 
\end{proof}

If the height of $T$ is larger than or equal to $s$, then the converse of Theorem \ref{MT4} also holds. 

\begin{prop}
Suppose that the height of $T$ is larger than or equal to $s$. 
For any modal formula $A$, if for any $\Sigma_1$ sentence $\sigma$, there exists an arithmetical interpretation $f$ such that $\PA + \neg \Con_T^{s+1} + \Con_T^s \vdash \sigma \leftrightarrow f_T(\Box A)$, then $\GL + \{\neg F_s\} \nvdash \Box A$ and $\GL + \{\neg F_s\} \nvdash \neg \Box A$. 
\end{prop}
\begin{proof}
If the height of $T$ is larger than $s$, then $T \nvdash \neg \Con_T^s$. 
By L\"ob's theorem, $T \nvdash \neg \Con_T^{s+1} \to \neg \Con_T^s$. 
Thus, the theory $T + \neg \Con_T^{s+1} + \Con_T^s$ is consistent. 
Also, if the height of $T$ is exactly $s$, then $\N \models \neg \Con_T^{s+1} \land \Con_T^s$. 
Thus, $\PA + \neg \Con_T^{s+1} + \Con_T^s$ is consistent. 
In either case, the theory $\PA + \neg \Con_T^{s+1} + \Con_T^s$ is consistent. 

Let $\sigma$ be any $\Sigma_1$ sentence independent of $\PA + \neg \Con_T^{s+1} + \Con_T^s$. 
Then, we have an arithmetical interpretation $f$ such that
\[
	\PA + \neg \Con_T^{s+1} + \Con_T^s \vdash \sigma \leftrightarrow f_T(\Box A).
\] 
Then, $\PA + \neg \Con_T^{s+1} + \Con_T^s$ proves neither $f_T(\Box A)$ nor $\neg f_T(\Box A)$. 
By Fact \ref{AS}.1, $\GL$ proves neither $\neg F_s \to \Box A$ nor $\neg F_s \to \neg \Box A$. 
\end{proof}

In the same way as in the proof of Theorem \ref{MT4}, the parameterized version of Theorem \ref{MT4} can also be proved, and hence we obtain the following theorem. 

\begin{thm}
Suppose that the height of $T$ is $s$. 
For any modal formula $A$, if $\GL + \{\neg F_s\} \nvdash \Box A$ and $\GL + \{\neg F_s\} \nvdash \neg \Box A$, then for any c.e.~set $X$, there exists an $v$-arithmetical interpretation $f$ such that $f_T(A)(v)$ weakly represents $X$ in $T$. 
\end{thm}
\begin{proof}
Let $\sigma(v)$ be a $\Sigma_1$ formula defining $X$ over $\N$. 
Then, from the parameterized version of Theorem \ref{MT4}, there exists an $v$-arithmetical interpretation $f$ such that 
\[
	\PA + \neg \Con_T^{s+1} + \Con_T^s \vdash \forall v\, \bigl(\sigma(v) \leftrightarrow f_T(\Box A)(v) \bigr). 
\]
Since the height of $T$ is $s$, $\N \models \neg \Con_T^{s+1} \land \Con_T^s$. 
Then, we have that $\N \models \forall v\, \bigl(\sigma(v) \leftrightarrow f_T(\Box A)(v) \bigr)$. 
It follows that $f_T(A)(v)$ weakly represents $X$ in $T$. 
\end{proof}

We close this section with two problems. 
In Section \ref{Classical}, we proved Theorem \ref{Thm1} by using Lemma \ref{Lem1}. 
On the other hand, in this section, we proved Theorem \ref{MT} that is a modal extension of Theorem \ref{Thm1} without using modal version of Lemma \ref{Lem1}. 
If a modal version of Lemma \ref{Lem1} is proved, then a proof of Theorem \ref{MT}, as well as the proof of our Theorem \ref{Thm1} in Section \ref{Classical}, would be substantially simplified. 

\begin{prob}
Can we prove a modal version of Lemma \ref{Lem1}?
\end{prob}

In Section \ref{Classical}, we proved Theorem \ref{Thm2} that is an improvement of Theorem \ref{Thm1}. 
Then, it is natural to ask the following problem. 

\begin{prob}
Can we extend Theorem \ref{Thm2} to the framework of modal propositional logic?
\end{prob}

\section{Rosser-type FGH theorems}\label{Rosser}

In this section, we investigate some variations of Rosser-type FGH theorems. 
Recall that $\Proof_T(x, y)$ is a natural formula saying that $y$ is a $T$-proof of $x$ and $\PR_T(x)$ is of the form $\exists y\, \Proof_T(x, y)$. 
Besides $\Proof_T(x, y)$, we consider formulas that witness $\PR_T(x)$. 
We say a formula $\Prf_T(x, y)$ is a \textit{proof predicate} of $T$ if $\Prf_T(x, y)$ satisfies the following conditions: 
\begin{enumerate}
	\item $\Prf_T(x, y)$ is a primitive recursive formula, 
	\item $\PA \vdash \forall x \, \bigl(\PR_T(x) \leftrightarrow \exists y \, \Prf_T(x, y) \bigr)$, 
	\item $\PA \vdash \forall x \, \forall x' \, \forall y\, \bigl(\Prf_T(x, y) \land \Prf_T(x', y) \to x = x' \bigr)$. 
\end{enumerate}

For each proof predicate $\Prf_T(x, y)$ of $T$, the $\Sigma_1$ formula $\PR_T^{\mathrm R}(x)$ defined by
\[
	 \PR_T^{\mathrm R}(x) : \equiv \exists y \, \bigl(\Prf_T(x, y) \land \forall z \, {<} \, y \, \neg \Prf_T(\dot{\neg} x, z) \bigr)
\]
is called the \textit{Rosser provability predicate} of $\Prf_T(x, y)$ or a Rosser provability predicate of $T$. 
Here $\dot{\neg} x$ is a primitive recursive term corresponding to a primitive recursive function calculating the G\"odel number of $\neg \varphi$ from the G\"odel number of an $\LA$-formula $\varphi$ such that $\PA$ proves $\dot{\neg} \gn{\psi} = \gn{\neg \psi}$ for each $\LA$-formula $\psi$. 
In view of witness comparison, we also introduce an auxiliary $\Sigma_1$ formula $\PR_T^{\R}(\dot{\neg}x)$ as follows: 
\[
	\PR_T^{\R}(\dot{\neg} x) : \equiv \exists y \, \bigl(\Prf_T(\dot{\neg} x, y) \land \forall z \, {\leq} \, y \, \neg \Prf_T(x, z) \bigr). 
\]
Then, for any $\LA$-formula $\varphi$, $\PA$ proves $\neg \bigl(\PR_T^{\mathrm R}(\gn{\varphi}) \land \PR_T^{\R}(\gn{\neg \varphi}) \bigr)$ and $\PR_T(\gn{\varphi}) \lor \PR_T(\gn{\neg \varphi}) \to \PR_T^{\mathrm R}(\gn{\varphi}) \lor \PR_T^{\R}(\gn{\neg \varphi})$.  

In a study of Rosser-type Henkin sentences, the following result was proved. 

\begin{fact}[Kurahashi {\cite[Theorem 3.5]{Kur}}]
For any $\Sigma_1$ sentence $\sigma$, the following are equivalent: 
\begin{enumerate}
	\item There exists a $\Sigma_1$ sentence $\sigma'$ such that $\PA \vdash \neg (\sigma \land \sigma')$ and $\PA \vdash \PR_T(\gn{\sigma}) \lor \PR_T(\gn{\sigma'}) \to \sigma \lor \sigma'$. 
	\item There exists a Rosser provability predicate $\PR_T^{\mathrm R}(x)$ of $T$ such that $\PA \vdash \sigma \leftrightarrow \PR_T^{\mathrm R}(\gn{\sigma})$. 
\end{enumerate}
\end{fact}

This fact can be seen as a kind of the FGH theorem for Rosser provability predicates because it deals with the equivalence $\sigma \leftrightarrow \PR_T^{\mathrm{R}}(\gn{\varphi})$ where $\varphi$ is in particular $\sigma$. 
Inspired by this fact, we prove the following theorem. 

\begin{thm}\label{MTR}
For any $\Sigma_1$ sentences $\sigma_0$ and $\sigma_1$, the following are equivalent: 
\begin{enumerate}
	\item $\PA \vdash \neg (\sigma_0 \land \sigma_1)$ and $\PA \vdash \neg \Con_T \to \sigma_0 \lor \sigma_1$. 
	\item There are a Rosser provability predicate $\PR_T^{\mathrm R}(x)$ of $T$ and an $\LA$-sentence $\varphi$ such that
	\[
		\PA \vdash \sigma_0 \leftrightarrow \PR_T^{\mathrm R}(\gn{\varphi}) \ \text{and}\ \PA \vdash \sigma_1 \leftrightarrow \PR_T^{\R}(\gn{\neg \varphi}).
	\] 
\end{enumerate}
\end{thm}
\begin{proof}
$(2 \Rightarrow 1)$: This implication follows from the properties of witness comparison formulas. 

$(1 \Rightarrow 2)$: 
We may assume that $\sigma_0$ and $\sigma_1$ are of the forms $\exists x \, \tau_0(x)$ and $\exists x \, \tau_1(x)$ for some $\Delta_0$ formulas $\tau_0(x)$ and $\tau_1(x)$, respectively. 
By the Fixed Point Lemma, let $\varphi$ be a $\Sigma_1$ sentence satisfying the following equivalence: 
\[
	\PA \vdash \varphi \leftrightarrow \exists y \, \bigl[(\Proof_T(\gn{\neg \varphi}, y) \lor \tau_0(y)) \land \forall z \, {\leq} \, y \, (\neg \Proof_T(\gn{\varphi}, z) \land \neg \tau_1(z)) \bigr].  
\]
Also, let $\varphi^*$ be the $\Sigma_1$ sentence 
\[
	\exists z \, \bigl[(\Proof_T(\gn{\varphi}, z) \lor \tau_1(z)) \land \forall y\, {<} \, z \, (\neg \Proof_T(\gn{\neg \varphi}, y) \land \neg \tau_0(y)) \bigr]. 
\]

\begin{lem}\label{L5}
$\PA \vdash \PR_T(\gn{\varphi}) \lor \PR_T(\gn{\neg \varphi}) \to \sigma_0 \lor \sigma_1$. 
\end{lem}
\begin{proof}
Let $\Prov_T^{\mathrm R}(x)$ be the Rosser provability predicate of $\Proof_T(x, y)$. 
Then, $\PA \vdash \PR_T(\gn{\varphi}) \lor \PR_T(\gn{\neg \varphi}) \to \Prov_T^{\mathrm R}(\gn{\varphi}) \lor \Prov_T^{\R}(\gn{\neg \varphi})$. 
By the definition of $\varphi^*$, 
\[
	\PA \vdash \Prov_T^{\mathrm R}(\gn{\varphi}) \land \neg \sigma_0 \to \varphi^*.
\]
Then, 
\[
	\PA \vdash \Prov_T^{\mathrm R}(\gn{\varphi}) \land \neg \sigma_0 \to \PR_T(\gn{\varphi^*})
\]
because $\varphi^*$ is a $\Sigma_1$ sentence. 
Since $\PA \vdash \neg (\varphi \land \varphi^*)$, we have
\begin{align}\label{L5fml}
	\PA \vdash \Prov_T^{\mathrm R}(\gn{\varphi}) \land \neg \sigma_0 \to \neg \Con_T.
\end{align}

By the choice of $\varphi$, 
\[
	\PA \vdash \Prov_T^{\R}(\gn{\neg \varphi}) \land \neg \sigma_1 \to \varphi,
\]
and we also obtain
\[
	\PA \vdash \Prov_T^{\R}(\gn{\neg \varphi}) \land \neg \sigma_1 \to \neg \Con_T.
\] 
From this with (\ref{L5fml}), 
\[
	\PA + \Con_T \vdash \Prov_T^{\mathrm R}(\gn{\varphi}) \lor \Prov_T^{\R}(\gn{\neg \varphi}) \to \sigma_0 \lor \sigma_1. 
\]
Thus, 
\[
	\PA + \Con_T \vdash \PR_T(\gn{\varphi}) \lor \PR_T(\gn{\neg \varphi}) \to \sigma_0 \lor \sigma_1.
\] 
Since $\PA + \neg \Con_T \vdash \sigma_0 \lor \sigma_1$, we conclude
\[
	\PA \vdash \PR_T(\gn{\varphi}) \lor \PR_T(\gn{\neg \varphi}) \to \sigma_0 \lor \sigma_1. \tag*{\mbox{\qedhere}}
\] 
\end{proof}

We proceed with the main proof. 
We recursively define a primitive recursive function $h$ and an increasing sequence $\langle k_i \rangle_{i \in \omega}$ of natural numbers simultaneously by referring to $T$-proofs in stages. 
The function $h$ will be defined to output all theorems of $T$ and the Rosser predicate of the proof predicate $x = h(y)$ will be a required one.
At the beginning of Stage $m$, the values of $k_0, \ldots, k_{m-1}, k_m$ and $h(0), \ldots, h(k_{m-1})$ have already been defined. 

Here we give our definition of the function $h$. 
In the definition, we identify each $\LA$-formula with its G\"odel number. 
First, let $k_0 = 0$. 

At Stage $m$, 
\begin{itemize}
	\item If $m$ is not a $T$-proof of any $\LA$-formula, then let $k_{m+1} : = k_m$ and go to Stage $m+1$. 
	\item If $m$ is a $T$-proof of an $\LA$-formula $\xi$ which is neither $\varphi$ nor $\neg \varphi$, then let $k_{m+1} : = k_m + 1$ and $h(k_m) : = \xi$, and go to the next stage. 
	\item If $m$ is a $T$-proof of an $\LA$-sentence $\xi$ which is either $\varphi$ or $\neg \varphi$, and $h$ already outputs at least one of $\varphi$ or $\neg \varphi$ before Stage $m$, then let $k_{m+1} : = k_m + 1$ and $h(k_m) : = \xi$, and go to Stage $m+1$. 
	\item If $m$ is a $T$-proof of either $\varphi$ or $\neg \varphi$, and $f$ does not output $\varphi$ and $\neg \varphi$ before Stage $m$, then we distinguish the following three cases: 
	\begin{description}
		\item [(a)] If $\exists z \, {\leq m} \, \tau_0(z)$ holds, then let $k_{m+1} : = k_m + 1$ and $h(k_m) : = \varphi$.
		\item [(b)] If $\exists z \, {\leq m} \, \tau_1(z)$ holds, then let $k_{m+1} : = k_m + 1$ and $h(k_m) : = \neg \varphi$. 
		\item [(c)] Otherwise, let $k_{m+1} : = k_m$. 
	\end{description}
	Go to the next Stage $m+1$. 
\end{itemize}
This completes our definition the function $h$. 
Since $\PA \vdash \neg (\sigma_0 \land \sigma_1)$, we have that $\PA$ proves that there is no $m$ satisfying both $\exists z \, {\leq} \, m\, \tau_0(z)$ and $\exists z \, {\leq} \, m\, \tau_1(z)$. 
Thus Cases (a) and (b) in the definition of $h$ are mutually exclusive.

We show that the formula $x = h(y)$ is a proof predicate of $T$. 
It suffices to show the following lemma: 

\begin{lem}\label{L6}
$\PA \vdash \forall x\, \bigl(\PR_T(x) \leftrightarrow \exists y \, (x = h(y)) \bigr)$. 
\end{lem}
\begin{proof}
We argue in $\PA$. 

$(\rightarrow)$: Suppose that $\xi$ is provable in $T$. 

If $\xi$ is neither $\varphi$ nor $\neg \varphi$, then for a $T$-proof $m$ of $\xi$, $h(k_m) = \xi$. 

If $\xi$ is either $\varphi$ or $\neg \varphi$, then by Lemma \ref{L5}, $\sigma_0 \lor \sigma_1$ holds. 
However, $\sigma_0$ and $\sigma_1$ cannot be true at the same time. 
We show that $\xi$ is output by $h$. 
We distinguish the following two cases: 

\begin{itemize}
	\item Case 1: $\sigma_0$ holds. \\
	Let $n$ be the least number such that $\tau_0(n)$ holds, and let $m$ be the least number such that $m \geq n$ and $m$ is a $T$-proof of either $\varphi$ or $\neg \varphi$. 
	Then, $h(k_m) = \varphi$ by the definition of $h$. 
	If $\xi$ is $\varphi$, we are done. 
	If $\xi$ is $\neg \varphi$, then let $m' > m$ be a $T$-proof of $\neg \varphi$. 
	Such an $m'$ exists because $\neg \varphi$ has infinitely many $T$-proofs. 
	Since $h$ already outputs $\varphi$ before Stage $m'$, by the definition of $h$, $h(k_{m'}) = \neg \varphi$. 

	\item Case 2: $\sigma_1$ holds. \\
	It is proved that $h$ outputs $\xi$ as in the proof of Case 1 by considering the least number $n$ such that $\tau_1(n)$ holds. 
\end{itemize}

$(\leftarrow)$: Suppose that $\xi$ is output by $h$. 
If $\xi$ is neither $\varphi$ nor $\neg \varphi$, then $h(k_m) = \xi$ implies that $m$ is a $T$-proof of $\xi$, and so $\xi$ is provable in $T$. 

If $\xi$ is $\varphi$ or $\neg \varphi$, then let $m$ be the least number such that $h(k_m)$ is either $\varphi$ or $\neg \varphi$. 
We show that $\xi$ is provable in $T$. 
We distinguish the following four cases: 

\begin{itemize}
	\item Case 1: $\xi$ is $\varphi$ and $h(k_m) = \varphi$. \\
	By the definition of $h$, $\exists z \, {\leq} \, m\, \tau_0(z)$ holds. 
	Then, $\sigma_0$ and $\neg \sigma_1$ hold. 
	Suppose that $\varphi$ is not provable in $T$, then $\varphi$ holds by the definition. 
	Since $\varphi$ is a $\Sigma_1$ sentence, it is provable in $T$, a contradiction. 
	Therefore, $\varphi$ is provable in $T$. 
	
	\item Case 2: $\xi$ is $\varphi$ and $h(k_m) = \neg \varphi$. \\
	Then, $h(k_{m'}) = \varphi$ for some $m' > m$. 
	Since $\neg \varphi$ is already output before Stage $m'$, we find that $m'$ is a $T$-proof of $\varphi$, and hence $\varphi$ is provable in $T$. 

	\item Case 3: $\xi$ is $\neg \varphi$ and $h(k_m) = \varphi$. \\
	Then, for some $T$-proof $m' > m$ of $\neg \varphi$, $h(k_{m'}) = \neg \varphi$. 
	Thus, $\neg \varphi$ is provable in $T$. 

	\item Case 4: $\xi$ is $\neg \varphi$ and $h(k_m) = \neg \varphi$. \\
	Then, $\exists z \, {\leq} \, m\, \tau_1(z)$ holds, and so $\sigma_1$ and $\neg \sigma_0$ hold. 
	Suppose that $\neg \varphi$ is not provable in $T$. 
	Then, $\varphi^*$ holds, and is provable in $T$. 
	Since $\varphi^* \to \neg \varphi$ is $T$-provable, $\neg \varphi$ is also $T$-provable. 
	This is a contradiction. 
	Therefore, $\neg \varphi$ is provable in $T$. \qedhere
\end{itemize}
\end{proof}

We resume the main proof. 
Let $\PR_h^{\mathrm R}(x)$ be the Rosser provability predicate of the proof predicate $x = h(y)$ of $T$. 
Finally, we show that $\PA$ proves $\sigma_0 \leftrightarrow \PR_h^{\mathrm R}(\gn{\varphi})$. 
A proof of $\PA \vdash \sigma_1 \leftrightarrow \PR_h^{\R}(\gn{\neg \varphi})$ is similar, and we omit it. 
We argue in $\PA$. 

$(\rightarrow)$: 
Suppose that $\sigma_0$ holds. 
Let $n$ be the least number such that $\tau_0(n)$ holds. 
Then, either $\varphi$ or $\varphi^*$ holds, and hence either one of them is provable in $T$. 
Then, either $\varphi$ or $\neg \varphi$ is $T$-provable. 
Let $m$ be the least number such that $m \geq n$ and $m$ is a $T$-proof of $\varphi$ or $\neg \varphi$. 
Then, $h$ does not output $\neg \varphi$ before Stage $m$, and $h(k_m) = \varphi$. 
This means $\PR_h^{\mathrm R}(\gn{\varphi})$ holds. 

$(\leftarrow)$: Suppose that $h(k_m) = \varphi$ and before Stage $m$, $h$ does not output $\neg \varphi$. 
For the least such an $m$, $\exists z \, {\leq} \, m\, \tau_0(z)$ holds, and thus $\sigma_0$ holds. 
\end{proof}

\begin{cor}\label{CorRos0}
For any $\Sigma_1$ sentence $\sigma$, the following are equivalent: 
\begin{enumerate}
	\item There exists a $\Sigma_1$ sentence $\sigma'$ such that $\PA \vdash \neg (\sigma \land \sigma')$ and $\PA \vdash \neg \Con_T \to \sigma \lor \sigma'$. 
	\item There exists a Rosser provability predicate $\PR_T^{\mathrm R}(x)$ of $T$ and an $\LA$-sentence $\varphi$ such that $\PA \vdash \sigma \leftrightarrow \PR_T^{\mathrm R}(\gn{\varphi})$. 
\end{enumerate}
\end{cor}
\begin{proof}
$(1 \Rightarrow 2)$: Immediate from Theorem \ref{MTR}. 

$(2 \Rightarrow 1)$: This implication is derived by letting $\sigma'$ be the $\Sigma_1$ sentence $\PR_T^{\R}(\gn{\neg \varphi})$. 
\end{proof}

The $(\PA + \Con_T)$-provable equivalence in the statement of the FGH theorem is equivalent to $\PA \vdash (\sigma \lor \neg \Con_T) \leftrightarrow \PR_T(\gn{\varphi})$. 
In this viewpoint, the following corollary seems to be a natural counterpart of the FGH theorem in terms of Rosser provability predicates.  

\begin{cor}\label{CorRos1}
For any $\Sigma_1$ sentence $\sigma$, there exist a Rosser provability predicate $\PR_T^{\mathrm R}(x)$ of $T$ and an $\LA$-sentence $\varphi$ such that
\[
	\PA \vdash (\sigma \lor \neg \Con_T) \leftrightarrow \PR_T^{\mathrm R}(\gn{\varphi}).
\] 
\end{cor}
\begin{proof}
For the $\Sigma_1$ sentences $\sigma \lor \neg \Con_T$ and $0=1$, we have that $\PA$ proves $\neg \bigl((\sigma \lor \neg \Con_T) \land 0=1 \bigr)$ and $\neg \Con_T \to (\sigma \lor \neg \Con_T) \lor 0=1$. 
By Corollary \ref{CorRos0}, there exist a Rosser provability predicate $\PR_T^{\mathrm R}(x)$ of $T$ and an $\LA$-sentence $\varphi$ such that $\PA \vdash (\sigma \lor \neg \Con_T) \leftrightarrow \PR_T^{\mathrm R}(\gn{\varphi})$. 
\end{proof}

Since $\PA + \Con_T \vdash \PR_T(\gn{\varphi}) \leftrightarrow \PR_T^{\mathrm R}(\gn{\varphi})$, the FGH theorem directly follows from Corollary \ref{CorRos1}. 

We show a version of the FGH theorem with respect to Rosser provability predicates corresponding to the representability of computable sets. 

\begin{cor}\label{CorRos2}
For any $\Delta_1(\PA)$ sentence $\delta$, there exist a Rosser provability predicate $\PR_T^{\mathrm R}(x)$ of $T$ and an $\LA$-sentence $\varphi$ such that
	\[
		\PA \vdash \delta \leftrightarrow \PR_T^{\mathrm R}(\gn{\varphi}) \ \text{and}\ \PA \vdash \neg \delta \leftrightarrow \PR_T^{\R}(\gn{\neg \varphi}).
	\] 
\end{cor}
\begin{proof}
Since $\delta$ is $\Delta_1(\PA)$, there exist $\Sigma_1$ sentences $\sigma_0$ and $\sigma_1$ such that $\PA \vdash \delta \leftrightarrow \sigma_0$ and $\PA \vdash \neg \delta \leftrightarrow \sigma_1$. 
Then, we have $\PA \vdash \neg (\sigma_0 \land \sigma_1)$ and $\PA \vdash \sigma_0 \lor \sigma_1$. 
By Theorem \ref{MTR}, there exists a Rosser provability predicate $\PR_T^{\mathrm R}(x)$ of $T$ and an $\LA$-sentence $\varphi$ such that
	\[
		\PA \vdash \sigma_0 \leftrightarrow \PR_T^{\mathrm R}(\gn{\varphi}) \ \text{and}\ \PA \vdash \sigma_1 \leftrightarrow \PR_T^{\R}(\gn{\neg \varphi}). \tag*{\mbox{\qedhere}}
	\] 
\end{proof}

Corollary \ref{CorRos2} says that if a $\Sigma_1$ sentence $\sigma$ is $\Delta_1(\PA)$, then there exist a Rosser provability predicate $\PR_T^{\mathrm R}(x)$ of $T$ and an $\LA$-sentence $\varphi$ such that $\PA \vdash \sigma \leftrightarrow \PR_T^{\mathrm R}(\gn{\varphi})$. 
Does this hold for all $\Sigma_1$ sentences?
By Corollary \ref{CorRos0}, this question is rephrased as follows: For any $\Sigma_1$ sentence $\sigma$, does there exists a $\Sigma_1$ sentence $\sigma'$ such that $\PA \vdash \neg (\sigma \land \sigma')$ and $\PA \vdash \neg \Con_T \to \sigma \lor \sigma'$?
We show that this is not the case. 

\begin{prop}
There exists a $\Sigma_1$ sentence $\sigma$ such that for all $\Sigma_1$ sentences $\sigma'$, neither $\PA \vdash \neg (\sigma \land \sigma')$ nor $\PA \vdash \neg \Con_T \to \sigma \lor \sigma'$. 
That is, for all Rosser provability predicates $\PR_T^{\mathrm R}(x)$ of $T$ and all $\LA$-sentences $\varphi$, $\PA \nvdash \sigma \leftrightarrow \PR_T^{\mathrm R}(\gn{\varphi})$. 
\end{prop}
\begin{proof}
Let $\sigma$ be a $\Sigma_1$ sentence which is $\Pi_1$-conservative over $\PA$ such that $\PA + \neg \Con_T \nvdash \sigma$. 
The existence of such a sentence is proved by Guaspari \cite[Theorem 2.4]{Gua} (see also Lindstr\"om \cite[Exercise 5.5 (b)]{Lin}). 
Suppose that for some $\Sigma_1$ sentence $\sigma'$, $\PA \vdash \neg (\sigma \land \sigma')$ and $\PA \vdash \neg \Con_T \to \sigma \lor \sigma'$. 
Then, $\PA + \sigma \vdash \neg \sigma'$, and hence $\PA \vdash \neg \sigma'$ by $\Pi_1$-conservativity. 
Thus, $\PA + \neg \Con_T \vdash \sigma$. 
This is a contradiction. 

\end{proof}

Guaspari and Solovay \cite{GS} introduced the logic $\mathbf{R}$ of witness comparison formulas, and also Shavrukov \cite{Sha} introduced the bimodal logic $\mathbf{GR}$ of usual and Rosser provability predicates. 
As in our observations in Section \ref{Modal}, it may also be possible to extend Theorem \ref{MTR} to the framework of modal logic via these logics. 
For example, for $\mathbf{R}$, we expect that the condition $\mathbf{R} + \{\Diamond^n \top \mid n \in \omega\} \nvdash \Box \Box A \to \Box A$ works well. 

\section*{Acknowledgements}

This work was partly supported by JSPS KAKENHI Grant Number JP19K14586. 
The author would like to thank Sohei Iwata, Haruka Kogure, and Yuya Okawa for their helpful comments. The author would also like to thank the anonymous referees for their valuable comments and suggestions. 
In particular, the current proofs of Theorems \ref{Thm1} and \ref{Thm2} using Lemmas \ref{Lem1} and \ref{Lem2} are based on the ideas of one of the referees, which made the proofs significantly simpler and easier to understand.

\bibliographystyle{plain}
\bibliography{ref}

\begin{thebibliography}{10}

\bibitem{AB}
Sergei~N. Artemov and Lev~D. Beklemishev.
\newblock Provability logic.
\newblock In D.~Gabbay and F.~Guenthner, editors, {\em Handbook of
  Philosophical Logic}, volume~13, pages 189--360. Springer, Dordrecht, 2nd
  edition, 2005.

\bibitem{Boo}
George {Boolos}.
\newblock {\em {The logic of provability}}.
\newblock Cambridge: Cambridge University Press, 1993.

\bibitem{EF}
Andrzej Ehrenfeucht and Solomon Feferman.
\newblock {Representability of recursively enumerable sets in formal theories}.
\newblock {\em {Archiv f\"ur Mathematische Logik und Grundlagenforschung}},
  5:37--41, 1961.

\bibitem{Gua}
David {Guaspari}.
\newblock {Partially conservative extensions of arithmetic}.
\newblock {\em {Transactions of the American Mathematical Society}},
  254:47--68, 1979.

\bibitem{GS}
David {Guaspari} and Robert~M. {Solovay}.
\newblock {Rosser sentences}.
\newblock {\em {Annals of Mathematical Logic}}, 16:81--99, 1979.

\bibitem{JD}
Giorgi {Japaridze} and Dick {de Jongh}.
\newblock {The logic of provability}.
\newblock In {\em Handbook of proof theory}, pages 475--546. Amsterdam:
  Elsevier, 1998.

\bibitem{Joo}
Joost~J. {Joosten}.
\newblock {Turing jumps through provability}.
\newblock In {\em Evolving computability. 11th conference on computability in
  Europe, CiE 2015, Bucharest, Romania, June 29 -- July 3, 2015. Proceedings},
  pages 216--225. Cham: Springer, 2015.

\bibitem{KK}
Makoto Kikuchi and Taishi Kurahashi.
\newblock Illusory models of {Peano} arithmetic.
\newblock {\em The Journal of Symbolic Logic}, 81(3):1163--1175, 2016.

\bibitem{Kur}
Taishi {Kurahashi}.
\newblock {Henkin sentences and local reflection principles for Rosser
  provability}.
\newblock {\em {Annals of Pure and Applied Logic}}, 167(2):73--94, 2016.

\bibitem{Kur18}
Taishi Kurahashi.
\newblock Provability logics relative to a fixed extension of {Peano}
  arithmetic.
\newblock {\em The Journal of Symbolic Logic}, 83(3):1229--1246, 2018.

\bibitem{Lin}
Per {Lindstr\"om}.
\newblock {\em {Aspects of incompleteness.}}, volume~10.
\newblock Natick, MA: Association for Symbolic Logic, 2003.

\bibitem{Seg}
Krister {Segerberg}.
\newblock {An essay in classical modal logic. Vol. 1, 2, 3}.
\newblock {Filosofiska Studier. No. 13. Uppsala: University of Uppsala.}, 1971.

\bibitem{Sha}
V.~Yu. {Shavrukov}.
\newblock {On Rosser's provability predicate}.
\newblock {\em {Zeitschrift f\"ur Mathematische Logik und Grundlagen der
  Mathematik}}, 37(4):317--330, 1991.

\bibitem{She}
John~C. Shepherdson.
\newblock {Representability of recursively enumerable sets in formal theories}.
\newblock {\em {Archiv f\"ur Mathematische Logik und Grundlagenforschung}},
  5:119--127, 1961.

\bibitem{Smo}
Craig {Smory\'nski}.
\newblock {Fifty years of self-reference in arithmetic}.
\newblock {\em {Notre Dame Journal of Formal Logic}}, 22:357--374, 1981.

\bibitem{Sol}
Robert~M. {Solovay}.
\newblock {Provability interpretations of modal logic}.
\newblock {\em {Israel Journal of Mathematics}}, 25:287--304, 1976.

\bibitem{Vis84}
Albert {Visser}.
\newblock {The provability logics of recursively enumerable theories extending
  Peano arithmetic at arbitrary theories extending Peano arithmetic}.
\newblock {\em {Journal of Philosophical Logic}}, 13:97--113, 1984.

\bibitem{Vis05}
Albert {Visser}.
\newblock {Faith \& falsity}.
\newblock {\em {Annals of Pure and Applied Logic}}, 131(1-3):103--131, 2005.

\end{thebibliography}

\end{document}